\documentclass[12pt]{amsart}

\usepackage{graphicx}
\usepackage{color}
\usepackage{hyperref}
\usepackage[margin=1in]{geometry}
\usepackage{enumerate}
\usepackage{tikz}
\usetikzlibrary{arrows,automata,positioning}

\newtheorem{theorem}{Theorem}
\numberwithin{theorem}{section}
\newtheorem{lemma}[theorem]{Lemma}
\newtheorem{proposition}[theorem]{Proposition}
\newtheorem{corollary}[theorem]{Corollary}

\theoremstyle{definition}
\newtheorem{definition}[theorem]{Definition}

\newtheorem{conjecture}[theorem]{Conjecture}

\theoremstyle{remark}
\newtheorem{remark}[theorem]{Remark}

\numberwithin{equation}{section}

\DeclareMathOperator{\spec}{Spec}
\DeclareMathOperator{\Dirspec}{DirSpec}

\newcommand{\Zeta}{\mathrm{Z}}

\begin{document}
	
\title{Spectral Properties of the Schreier Graphs of the basilica Group}

\author[Ambrose, Dunham, Morris, Rogers, Teplyaev]{Kyle Ambrose, Noah Dunham, Michael Morris, Luke~G. Rogers, Alexander Teplyaev}

\thanks{Research supported by NSF DMS 2349433 and the Simons Foundation.}                          
    
\begin{abstract}
We study the spectral properties of Laplacians on the Schreier graphs $\Gamma_n$ of the basilica group, the iterated monodromy group of the polynomial $z^2 - 1$, which is an important example in the theory of self-similar, amenable but not elementarily amenable, automaton groups. Building heavily on results by Brzoska, Jarvis, George, Rogers and Teplyaev~\cite{brzoska} about certain subgraphs of the basilica graphs, we develop a new recursive framework for computing the characteristic polynomials of these Laplacians. Our analysis reveals a simple underlying dynamical system and proves approximation results for the Kesten-von Neumann-Serre (KNS) spectral measure.
\end{abstract}

\date{\today}
\maketitle

\section{Introduction}
The basilica group was introduced by Grigorchuk and \.Zuk in~\cite{GZ1} and has become a foundational example in the study of self-similar automata groups. The spectral properties of its sequence $\Gamma_n$ of Schreier graphs were first studied in~\cite{GZ2} using techniques from~\cite{BG2000a,BG2000b}, where it was found that the spectral problem for a family of weighted graph Laplacians could be described by a two-dimensional dynamical system that intertwined the spectral values and weights.  Subsequent work established that it is amenable~\cite{MR2176547} and can be viewed as the iterated monodromy group of $z\mapsto z^2-1$ in the sense of Nekrashevych~\cite{Nekbook}. The structure of its orbital Schreier graphs, which are pointed Gromov Hausdorff limits of the $\Gamma_n$, was elucidated in~\cite{Nagnibeda}.

In recent work~\cite{brzoska} two of the authors and their collaborators were able to give a dynamical description of the spectra  of a sequence of graphs related to the basilica Schreier graphs and use these to analyze the limiting spectral measure, which can be identified as a version of the Kesten-von Neumann-Serre (KNS) measure introduced in~\cite{GZ04}. Dang, Grigorchuk and Lyubich have also announced a generalization of the techniques from~\cite{dang2020self} to the case of the basilica Schreier graphs.  Both contribute to the increasing family of methods applicable to self-similar groups, fractal structures and aperiodic order, including~\cite{Kaimanovich03,Kaimanovich05,Kaimanovich09,NT,MR3558157,GLN1,GLN2,grigorchuk2020integrable}. 

The purpose of the present work is to extend the methods of~\cite{brzoska} to the basilica Schreier graphs $\Gamma_n$. In Section~\ref{section:GraphsDefinitions}  we  introduce the graphs $\Gamma_n$ and the related family of graphs $G_n$ that were studied in~\cite{brzoska}. The topological structure of Laplacian eigenfunctions on these graphs is explored in Section~\ref{sec:topologyofspectrum}. This structure permits us to count the multiplicities of factors of the characteristic polynomials in Section~\ref{sec:dynamics} and then develop a dynamical description of these factors. Our main results in this section, Theorem~\ref{thm:mainrecursion} and Corollary~\ref{cor:0seriesfromzeta=2}, allow us to describe all eigenvalues using the dynamical system introduced in~\cite{brzoska}. Consequences for the structure of the spectrum and the limit of the spectral measure are discussed in Section~\ref{sec:SpectralMeasure}. In particular we give a formula for the KNS measure as a sum of point masses (Theorem~\ref{thm:KNSSpectralMeasureofGamma}), estimate the rate of convergence of the spectral counting measures to the KNS measure (Theorem~\ref{thm:SpectralMeasureIsAllSupportedOnHighLevel}, but see also Remark~\ref{rem:dynamicaldegreecontrolsconvergence}), and establish that the eigenvalues of the basilica Schreier graphs accumulate precisely on the support of the KNS measure, which is a Cantor set (Theorem~\ref{thm:accumsetofevals}).

\section{The graphs $\Gamma_n$ and $G_n$ and their Laplacians}\label{section:GraphsDefinitions}

The Schreier graphs $\Gamma_n$ of the self-similar basilica group may be constructed in several different ways.  We give a recursive definition via a sequence of subgraphs $G_n$ that has minimal prerequisite knowledge and refer to~\cite{GZ1,GZ2,Nagnibeda} for many geometric and algebraic properties that flow more naturally from the construction of the $\Gamma_n$ via the basilica group.

\begin{definition}\label{def:G_nConstruction}
The graphs $G_n$, $n \geq 2$, may be constructed inductively from a copy of $G_{n-1}$ and two copies of $G_{n-2}$ in which four boundary
points are identified to a single vertex as shown in Figure~\ref{fig:G_nGluing}.  The initial graphs $G_0$, $G_1$, and $G_2$ are shown in Figure \ref{fig:InitialG_n}. We write $\partial G_n$ for the boundary points of $G_n$.
\end{definition}

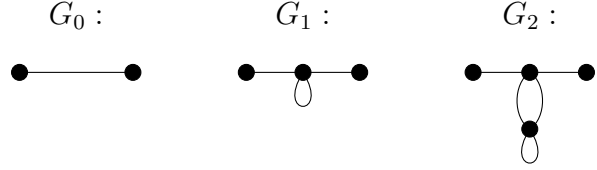
\begin{figure}
    \begin{tikzpicture}[scale = 1.5]
        \draw (0.5,0.5) node {\(G_0:\)};
        
        \filldraw [black] (0,0) circle (2pt);
        \filldraw [black] (1,0) circle (2pt);

        \draw (0,0) .. controls (0.5,0) .. (1,0);

        \draw (2.5,0.5) node {\(G_1:\)};

        \filldraw [black] (2,0) circle (2pt);
        \filldraw [black] (2.5,0) circle (2pt);
        \filldraw [black] (3,0) circle (2pt);

        \draw (2,0) .. controls (2.5,0) .. (3,0);
        \draw (2.5,0) .. controls (2.25,-0.4) and (2.75,-0.4) .. (2.5,0);

        \draw (4.5,0.5) node {\(G_2:\)};
        
        \filldraw [black] (4,0) circle (2pt);
        \filldraw [black] (4.5,0) circle (2pt);
        \filldraw [black] (4.5,-0.5) circle (2pt);
        \filldraw [black] (5,0) circle (2pt);

        \draw (4,0) .. controls (4.5,0) .. (5,0);
        \draw (4.5,0) .. controls (4.35,-0.1) and (4.35,-0.4) .. (4.5,-0.5);
        \draw (4.5,0) .. controls (4.65,-0.1) and (4.65,-0.4) .. (4.5,-0.5);
        \draw (4.5,-0.5) .. controls (4.25,-0.9) and (4.75,-0.9) .. (4.5,-0.5);
    \end{tikzpicture}
    \caption{The graphs $G_0$, $G_1$, $G_2$.}
    \label{fig:InitialG_n}
\end{figure}

\begin{figure}
    \begin{tikzpicture}[scale = 1.1]
        \draw (2,4.5) node {\(G_{n-2}\)};
        
        \filldraw [black] (6,4.5) circle (2pt);
        \filldraw [black] (4.5,4.5) circle (2pt);
        \filldraw [black] (3,4.5) circle (2pt);

        \draw (6,4.5) .. controls (4.5,4.5) .. (3,4.5);
        \draw (4.5,4.5) .. controls (4,3.5) and (5,3.5) .. (4.5,4.5);

        \draw (12,4.5) node {\(G_{n-2}\)};
        
        \filldraw [black] (8,4.5) circle (2pt);
        \filldraw [black] (9.5,4.5) circle (2pt);
        \filldraw [black] (11,4.5) circle (2pt);

        \draw (8,4.5) .. controls (9.5,4.5) .. (11,4.5);
        \draw (9.5,4.5) .. controls (9,3.5) and (10,3.5) .. (9.5,4.5);
        
        \draw (6,2) node {\(G_{n-1}\)};
        
        \filldraw [black] (6.5,4) circle (2pt);
        \filldraw [black] (7.5,4) circle (2pt);
        \filldraw [black] (7,3) circle (2pt);
        \filldraw [black] (7,2) circle (2pt);

        \draw (6.5,4) .. controls (6.5,3.5) and (6.75,3) .. (7,3);
        \draw (7.5,4) .. controls (7.5,3.5) and (7.25,3) .. (7,3);
        \draw (7,3) .. controls (6.75,2.7) and (6.75,2.3) .. (7,2);
        \draw (7,3) .. controls (7.25,2.7) and (7.25,2.3) .. (7,2);
        \draw (7,2) .. controls (6.5,1) and (7.5,1) .. (7,2);

        \filldraw [black] (7,5) circle (2pt);

        \draw[->][dashed] (6.1,4.55) -- (6.9,4.95);
        \draw[->][dashed] (7.9,4.55) -- (7.1,4.95);
        \draw[->][dashed] (6.55,4.1) -- (6.95,4.9);
        \draw[->][dashed] (7.45,4.1) -- (7.05,4.9);
    \draw (7,5.25) node {$v$};
    
    \end{tikzpicture}
    \caption{The graph $G_n$ constructed from a copy of $G_{n-2}$ and two copies of $G_{n-1}$ at a gluing point $v$.} 
    \label{fig:G_nGluing} 
\end{figure}
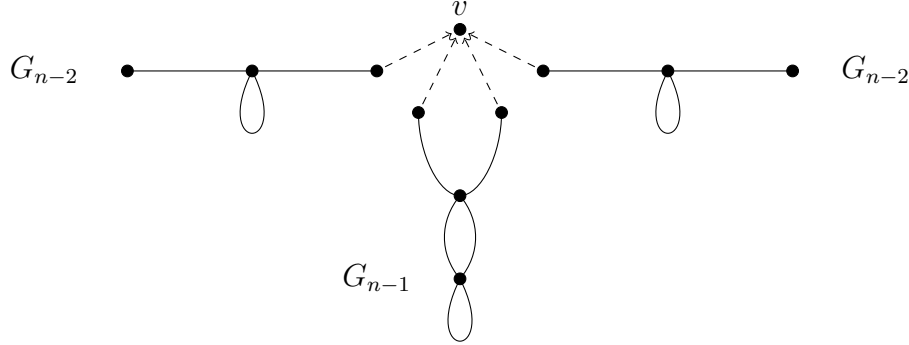

\begin{definition}\label{def:Gamma_nConstruction}
    The graphs $\Gamma_n$, $n\geq1$, are constructed by gluing a copy of \(G_n\) and a copy of \(G_{n-1}\) as shown in Figure~\ref{fig:2WayGluing}. The four boundary vertices of $G_n$ and $G_{n-1}$  are identified to a single vertex. Equivalently, the graph $\Gamma_n$ is constructed by gluing together two copies of $G_{n-1}$ and two copies of $G_{n-2}$ as shown in Figure~\ref{fig:4WayGluing}, where the graphs are connected at two gluing vertices $u$ and $v$.  The exception is $\Gamma_0$, which consists of two loops glued at a single point.
\end{definition}

Note that the labelling in the figures is consistent:  viewing the copy of $G_n$ in Figure~\ref{fig:2WayGluing} as arising from the gluing in Figure~\ref{fig:G_nGluing} produces Figure~\ref{fig:4WayGluing}, with the vertex $v$ being the gluing point in $G_n$ in both diagrams.

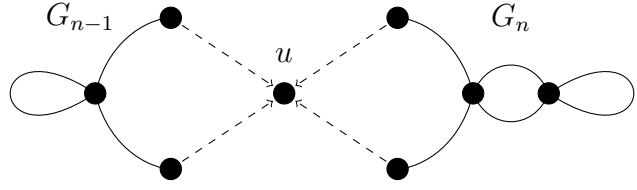
\begin{figure}
    \begin{tikzpicture}[scale=2]
        \draw (-1.5,0) .. controls (-1.4,0.4) and (-1.1,0.5) .. (-1,0.5);
        \draw (-1.5,0) .. controls (-1.4,-0.4) and (-1.1,-0.5) .. (-1,-0.5);
        \draw (-1.5,0) .. controls (-2.25,0.5) and (-2.25,-0.5) .. (-1.5,0);

        \draw (1,0) .. controls (0.9,0.4) and (0.6,0.5) .. (0.5,0.5);
        \draw (1,0) .. controls (0.9,-0.4) and (0.6,-0.5) .. (0.5,-0.5);
        \draw (1,0) .. controls (1.1,0.25) and (1.4,0.25) .. (1.5,0);
        \draw (1,0) .. controls (1.1,-0.25) and (1.4,-0.25) .. (1.5,0);
        \draw (1.5,0) .. controls (2.25,0.5) and (2.25,-0.5) .. (1.5,0);
        
        \filldraw [black] (-0.25,0) circle (2pt);

        \filldraw [black] (-1,0.5) circle (2pt);
        \filldraw [black] (-1,-0.5) circle (2pt);
        \filldraw [black] (-1.5,0) circle (2pt);

        \filldraw [black] (0.5,0.5) circle (2pt);
        \filldraw [black] (0.5,-0.5) circle (2pt);
        \filldraw [black] (1,0) circle (2pt);
        \filldraw [black] (1.5,0) circle (2pt);
        
        \draw (-1.6,0.5) node {\(G_{n-1}\)};
        \draw (1.25,0.5) node {\(G_n\)};
        \draw (-0.25,0.25) node {$u$};

         \draw[->][dashed] (-0.925,0.45) -- (-0.325,0.05);
         \draw[->][dashed] (-0.925,-0.45) -- (-0.325,-0.05);
         \draw[->][dashed] (0.425,0.45) -- (-0.175,0.05);
         \draw[->][dashed] (0.425,-0.45) -- (-0.175,-0.05);
    \end{tikzpicture}
    \caption{The graph $\Gamma_n$ constructed from a copy of $G_n$, and a copy of $G_{n-1}$ at a gluing point $u$.}
    \label{fig:2WayGluing}
\end{figure}

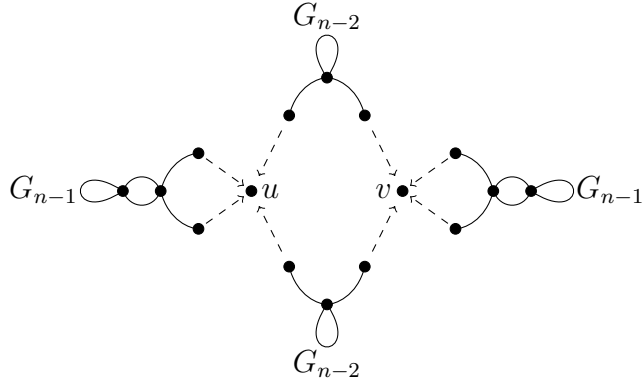
\begin{figure}
    \begin{tikzpicture}[scale = 1]
        \filldraw [black] (-1,0) circle (2pt);
        \filldraw [black] (1,0) circle (2pt);

        \filldraw [black] (-1.7,0.5) circle (2pt);
        \filldraw [black] (-1.7,-0.5) circle (2pt);
        \filldraw [black] (-2.2,0) circle (2pt);
        \filldraw [black] (-2.7,0) circle (2pt);

        \filldraw [black] (1.7,0.5) circle (2pt);
        \filldraw [black] (1.7,-0.5) circle (2pt);
        \filldraw [black] (2.2,0) circle (2pt);
        \filldraw [black] (2.7,0) circle (2pt);

        \filldraw [black] (-0.5,1) circle (2pt);
        \filldraw [black] (0.5,1) circle (2pt);
        \filldraw [black] (0,1.5) circle (2pt);
        
        \filldraw [black] (-0.5,-1) circle (2pt);
        \filldraw [black] (0.5,-1) circle (2pt);
        \filldraw [black] (0,-1.5) circle (2pt);

        \draw (-2.2,0) .. controls (-2.1,0.4) and (-1.8,0.5) .. (-1.7,0.5);
        \draw (-2.2,0) .. controls (-2.1,-0.4) and (-1.8,-0.5) .. (-1.7,-0.5);
        \draw (-2.2,0) .. controls (-2.3,0.25) and (-2.6,0.25) .. (-2.7,0);
        \draw (-2.2,0) .. controls (-2.3,-0.25) and (-2.6,-0.25) .. (-2.7,0);
        \draw (-2.7,0) .. controls (-3.45,0.5) and (-3.45,-0.5) .. (-2.7,0);

        \draw (2.2,0) .. controls (2.1,0.4) and (1.8,0.5) .. (1.7,0.5);
        \draw (2.2,0) .. controls (2.1,-0.4) and (1.8,-0.5) .. (1.7,-0.5);
        \draw (2.2,0) .. controls (2.3,0.25) and (2.6,0.25) .. (2.7,0);
        \draw (2.2,0) .. controls (2.3,-0.25) and (2.6,-0.25) .. (2.7,0);
        \draw (2.7,0) .. controls (3.45,0.5) and (3.45,-0.5) .. (2.7,0);

        \draw (-0.5,1) .. controls (-0.4,1.4) and (-0.1,1.5) .. (0,1.5);
        \draw (0.5,1) .. controls (0.4,1.4) and (0.1,1.5) .. (0,1.5);
        \draw (0,1.5) .. controls (0.5,2.25) and (-0.5,2.25) .. (0,1.5);

        \draw (-0.5,-1) .. controls (-0.4,-1.4) and (-0.1,-1.5) .. (0,-1.5);
        \draw (0.5,-1) .. controls (0.4,-1.4) and (0.1,-1.5) .. (0,-1.5);
        \draw (0,-1.5) .. controls (0.5,-2.25) and (-0.5,-2.25) .. (0,-1.5);
        
        \draw[->][dashed] (-1.6,0.42857) -- (-1.1,0.071428);
        \draw[->][dashed] (-1.6,-0.42857) -- (-1.1,-0.071428);
        \draw[->][dashed] (1.6,0.42857) -- (1.1,0.071428);
        \draw[->][dashed] (1.6,-0.42857) -- (1.1,-0.071428);
        \draw[->][dashed] (-0.6,0.8) -- (-0.9,0.2);
        \draw[->][dashed] (-0.6,-0.8) -- (-0.9,-0.2);
        \draw[->][dashed] (0.6,0.8) -- (0.9,0.2);
        \draw[->][dashed] (0.6,-0.8) -- (0.9,-0.2);
        
        \draw (-3.75,0) node {\(G_{n-1}\)};
        \draw (3.75,0) node {\(G_{n-1}\)};
        \draw (0,2.3) node {\(G_{n-2}\)};
        \draw (0,-2.3) node {\(G_{n-2}\)};

        \draw (-0.75,0) node {$u$};
        \draw (0.75,0) node {$v$};
    \end{tikzpicture}
    \caption{The graph $\Gamma_n$ constructed from two copies of $G_{n-1}$, and two copies of $G_{n-2}$, attached at gluing points $u$ and $v$.}
    \label{fig:4WayGluing}
\end{figure}

\begin{definition}\label{def:GraphLaplacian}
We define a Laplacian operator $\Delta(H)$ for a graph $H$ in the usual manner.  Let $\ell^2(H)$ denote the functions $\mathbb{R}^{H}$ with $L^2$ norm with respect to the counting measure on the vertex set. For vertices $x,y$ of $H$ let $c_{xy}$ be the number of edges joining $x$ and $y$ and note that $c_{xy}\in\{0,1,2\}$ for the graphs considered in this paper.
The Laplacian on $\ell^2(H)$ is
    \begin{equation}\label{eqn:defnofLap}
	\Delta(H)  f(x) =  \sum_y c_{xy}(f(x)- f(y)).
	\end{equation}
\end{definition}
$\Delta(H)$ is self-adjoint; it is irreducible, because the graphs we are using are connected; and it is non-negative definite, because  $\sum_x f(x)\Delta(H)f(x)=\frac12\sum_{x,y} c_{xy} (f(x)-f(y))^2$.

We will also make substantial use of the Dirichlet Laplacian, which is given by \eqref{eqn:defnofLap} but with domain the functions $\{f\in \mathbb{R}^{H}:f|_{\partial H}=0\}$.

\begin{definition}\label{def:SpecDefinition} For a graph $H$, the set of eigenvalues of $\Delta(H)$ is denoted
$\spec(\Delta(H))$. The set of Dirichlet eigenvalues for the Laplacian on $H$ is denoted $\Dirspec(\Delta(H))$.
\end{definition}

\begin{definition}\label{def:BornEigenvalues}
An eigenvalue $\lambda$ is born at level $n$ for $\Delta(\Gamma_n)$ if $\lambda  \in \spec(\Delta(\Gamma_n))$, but $ \lambda \not\in \spec(\Delta(\Gamma_m))$ for any $m < n$. Similarly, an eigenvalue $\lambda$ is born at level $n$ for $\Delta(G_n)$ if $\lambda  \in \spec(\Delta(G_n))$, but $ \lambda \not\in \spec(\Delta(G_m))$ for any $m < n$. 
\end{definition}

\begin{lemma}[Size of \(\Gamma_n\)]\label{lem:SizeofGamma}
    At level \(n\), \(\Gamma_n\) has \(2^n\) vertices.
\end{lemma}
\begin{proof}
    From Definition \ref{def:Gamma_nConstruction} it must be that \(\# \Gamma_n=\# G_n+\# G_{n-1}-3\) where \(\#\) is the counting measure on the number of vertices in the graph. By Lemma 2.5 in \cite{brzoska}, $\# G_n=\frac16\bigl(2^{2+n}+(-1)^{1+n}+9\bigr)$. Direct computation gives the result.
\end{proof}


\section{Topological structure of the spectrum of the graphs $\Gamma_n$}\label{sec:topologyofspectrum}

One of the main results of~\cite{brzoska} was a decomposition of the Dirichlet spectrum of the graphs $G_n$ according to the topology of the graphs. Specifically, it was shown that there is a correspondence between bases for the Dirichlet eigenspaces and families of loops in the graphs and that these families of loops have a self-similar structure that enables the dimension of the eigenspaces  to be computed.  In this section we generalize this description to the graphs $\Gamma_n$ by relating eigenfunctions of the Laplacian on the graphs $G_n$ and on the graphs $\Gamma_n$ and applying results from~\cite{brzoska}. In consequence we are able to identify the eigenspaces of $\Delta(\Gamma_n)$ that correspond to families of loops seen in the graphs $G_n$ and therefore count multiplicities of these eigenspaces. This result is in Theorem~\ref{thm:2seriesefns}.

 We are also able to describe a family of one-dimensional eigenspaces that occur for $\Delta(\Gamma_n)$ but not for $G_n$, and that correspond to the central loop in Figure~\ref{fig:4WayGluing}. The corresponding eigenvalues are called the $0$-series, and the fact that they are non-trivial on this loop is apparent from the definition (Definition~\ref{def:02Series}), which requires $f(u)f(v)\neq0$ in Figure~\ref{fig:4WayGluing}. The fact that the eigenvalues are simple is proved in Lemma~\ref{lem:0seriessimple}.

\subsection{Dirichlet spectrum and eigenfunctions of $G_n$}

The following results are from Section~3 of~\cite{brzoska}.  For $f$ on $G_n$ we write $\partial f(x)$ for the Neumann derivative of $f$ at a boundary point $x$, meaning $\partial f(x)=f(x)-f(y)$ where $y$ is the unique neighbor of $x$. The significance of the Neumann derivative is that when copies of the graphs are glued by identifying boundary points, the Laplacian at the gluing point is the sum of the Neumann derivatives at the identified boundary points of the glued graphs. In what follows, the gluing point $v$ is as in Figure~\ref{fig:G_nGluing}.

We consider $f$ that is a solution of $\Delta(G_n)f=\lambda f$ on $G_n\setminus\partial G_n$.  If $f=0$ on $\partial G_n$ it is a Dirichlet eigenfunction and $\lambda\in\Dirspec(\Delta(G_n))$.    If, in addition,  $\partial f=0$ on $\partial G_n$ we call $f$ a Dirichlet-Neumann eigenfunction.

\begin{theorem}[\protect{From~\cite{brzoska} Section~3}]\label{thm:Gnefntheorem}
Suppose $\Delta(G_n)f=\lambda f$ on $G_n\setminus\partial G_n$.
\begin{enumerate}[(1)]
\item\label{item:Gnefn} If $f=0$ on $\partial G_n$, so $f$ is a Dirichlet eigenfunction and $\lambda\in\Dirspec{\Delta(G_n)}$, let $m\leq n$ be the level at which $\lambda$ is born. Then:
\begin{enumerate}[(i)]
\item\label{item:Gnefnsimple} If $n=m$ then $\lambda$ is simple,  $f(v)\neq0$ and the Neumann derivatives at the two boundary points are equal and non-zero. We call the resulting eigenfunction $h_n$. Conversely, if $f(v)\neq0$ then $n=m$.
\item\label{item:Gnefnn-modd} If $m<n$ and $n-m$ is odd then $f$ is Dirichlet-Neumann. There are no such eigenfunctions for $m=n-1$.
\item\label{item:DnotNefn} If $m< n$ and $n-m$ is even then there is a one-dimensional space of eigenfunctions that are Dirichlet but not Neumann. A basis element may be constructed by decomposing the shortest path between the points of $\partial G_n$ into copies of $G_m$ and arranging copies of $h_m$ with alternating signs along this path; the function is zero at all other points. These eigenfunctions vanish at $v$ and have non-zero Neumann derivatives with equal and opposite signs at the points of $\partial G_n$. When $n=m+2$ there are no Dirichlet-Neumann eigenfunctions.
\item\label{item:DNefns} If $m<n-2$ there are Dirichlet-Neumann eigenfunctions $f$ that arise by requiring the restriction of $f$ to each of the copies of $G_{n-1}$, $G_{n-2}$ that are glued at $v$ to either vanish or be Dirichlet-Neumann. If $n-m$ is odd then there are also Dirichlet-Neumann eigenfunctions that vanish on both copies of $G_{n-2}$ and coincide with one of the Dirichlet but not Neumann eigenfunctions from item~\eqref{item:DnotNefn} on the copy of $G_{n-1}$.  These two constructions give all of the Dirichlet-Neumann eigenfunctions.
\end{enumerate}
\item \label{item:Gnnonefn} If $f$ is not a Dirichlet eigenfunction but $\lambda\in\Dirspec(\Delta(G_m))$ is born at level $m$ then:
\begin{enumerate}[(i)]
\item\label{item:notGnefn1} If $m=n$ then the values of $f$ at the points of $\partial G_n$ are non-zero, equal in magnitude and opposite in sign.
\item\label{item:notGnefn2} If $m=n-1$ then $f$ is from a two-dimensional space with basis as follows. Each basis function vanishes at $v$ and both the function and its Neumann derivative vanish at one point of $\partial G_n$. At the other point of $\partial G_n$ the function has value $1$ and a non-zero Neumann derivative. In particular, $f(v)=0$ in Figure~\ref{fig:G_nGluing} for $G_n$.
\item\label{item:notGnefn4} If $n-m\geq 3$ is odd there is a two-dimensional space of solutions. A basis is obtained by recognizing that at each boundary point of $G_n$ one has a copy of $G_{m+1}$, and that the function in~\eqref{item:notGnefn2} can be placed on this copy so the value $1$ occurs at the boundary point. At the other boundary point of the copy of $G_{m+1}$ this function has both Dirichlet and Neumann conditions so can be extended by zero to all of $G_n$ to satisfy $\Delta f=\lambda f$ on $G_n\setminus \partial G_n$. In particular, $f(v)=0$ in Figure~\ref{fig:G_nGluing} for $G_n$.
\item\label{item:notGnefn3} If $m=n-2$ then $f$ has the same value at both points of $\partial G_n$ and its value at $v$ is the negative of its value on $\partial G_n$.
\item\label{item:notGnefn5} If $n-m\geq 4$ is even then there are no functions of this type.
\end{enumerate}
\end{enumerate}
\end{theorem}

\begin{proof}
We give more specific references for the individual items. All references are to results in~\cite{brzoska}. Items~\eqref{item:Gnefnsimple} and~\eqref{item:notGnefn1} are from Proposition~3.5 and Figure~8. The remaining subitems in~\eqref{item:Gnefn} are from Theorem~3.11.  Item~\eqref{item:notGnefn2} is from Lemma~3.7 and Figure~9.  Item~\eqref{item:notGnefn3} is from Lemma~3.9 and Figure~10, and Items~\eqref{item:notGnefn4} and~\eqref{item:notGnefn5} are from Lemma~3.10. 
\end{proof}


\subsection{Spectrum and eigenfunctions on $\Gamma_n$}  

We relate eigenfunctions between levels of $\Gamma_n$ and between $\Gamma_n$ and $G_m$ graphs in order to establish structural features that will allow us to count the multiplicities of eigenvalues.  For this purpose we frequently use the decompositions in Figures~\ref{fig:G_nGluing}, \ref{fig:2WayGluing} and \ref{fig:4WayGluing}. In particular, we refer to the points $u$ and $v$ in Figure~\ref{fig:4WayGluing} as gluing points and note that the gluing point $u$ in Figure~\ref{fig:4WayGluing} is the same point as $u$ in Figure~\ref{fig:2WayGluing} and that the point $v$ in Figure~\ref{fig:4WayGluing} is the same point as $v$ in the copy of $G_n$ as shown in Figure~\ref{fig:G_nGluing}.

\begin{lemma}\label{lem:EigenvaluesPersist}
If $\lambda \in \spec(\Delta(\Gamma_n))$ then $\lambda \in \spec(\Delta(\Gamma_{n+1}))$.
\end{lemma}
\begin{proof}
From Figure~\ref{fig:2WayGluing} take the restriction of the eigenfunction $f$ on $\Gamma_n$ to the copies of $G_{n-1}$ and $G_n$ to define functions on these graphs. Then construct $\Gamma_{n+1}$ from two copies of $G_{n-1}$ and two copies of $G_{n}$ as in Figure~\ref{fig:4WayGluing} and copy these functions to the new graph. The result satisfies the eigenfunction equation: at points other than $u$ and $v$ in Figure~\ref{fig:4WayGluing} this is immediate, while at $u$ and $v$ we need only observe that the gluing has the same values and Neumann derivatives at these points as $f$ does at the point $u$ of Figure \ref{fig:2WayGluing}.
\end{proof}
\begin{remark}\label{rem:EigenvaluesPersist}
An observation about this proof that will be useful later is that the function constructed on $\Gamma_{n+1}$ has the same value at $u$ and at $v$ in Figure~\ref{fig:4WayGluing} as the eigenfunction on $\Gamma_n$ had at the point  $u$ of Figure \ref{fig:2WayGluing}.
\end{remark}

A partial converse is available by recognizing the role of symmetry in this construction.  Together, the proofs give a one-to-one correspondence between eigenfunctions on $\Gamma_{n-1}$ and eigenfunctions on $\Gamma_n$ that are symmetric under a reflection.

\begin{corollary}\label{cor:Gamma_ntoGamma_n-1}
If $n\geq 2$ and $\lambda\in\spec(\Delta(\Gamma_n))$ has an eigenfunction $f$ that is symmetric under the reflection across the central vertical line of symmetry in Figure~\ref{fig:4WayGluing} then $\lambda\in\spec(\Delta(\Gamma_{n-1}))$.
\end{corollary}
\begin{proof}
Working in Figure~\ref{fig:4WayGluing}, restrict $f$ to the copy of $G_{n-1}$ that contains $u$ and to one of the copies of $G_{n-2}$, then identify $u$ and $v$ to obtain $\Gamma_{n-1}$ as in Figure~\ref{fig:2WayGluing}. Observe that $f(u)=f(v)$ ensures $f$ is well-defined on $\Gamma_{n-1}$ and that the Laplacian at this point is the same as it was at $u$ in $\Gamma_n$. It follows that this is an eigenfunction on $\Gamma_{n-1}$.
\end{proof}

\begin{lemma}\label{lem:Dirspectospec}
If $\lambda\in\Dirspec(\Delta(G_m))$ is born at level $m$ then, for all $n\geq m+2$, $\lambda\in\spec(\Delta(\Gamma_n))$ and has an eigenfunction that vanishes at both gluing points $u$ and $v$ in Figure~\ref{fig:4WayGluing}.
\end{lemma}
\begin{proof}
From Theorem~\ref{thm:Gnefntheorem}~\eqref{item:Gnefnsimple} the hypothesis gives a one-dimensional eigenspace spanned by a function $h_m$ on $G_m$.  Now construct a function $f$ on  $\Gamma_{m+2}$ by viewing the graph in the manner of Figure~\ref{fig:4WayGluing} as two copies of $G_m$ and two of $G_{m+1}$. Set $f\equiv0$ on the copies of $G_{m+1}$,  $f=h_m$ on one copy of $G_m$ and $f=-h_m$ on the other copy of $G_m$.  Evidently $\Delta f=\lambda f$ at points other than the gluing points $u$ and $v$. However, at these points the Neumann derivatives from the copies of $h_m$ and $-h_m$ are equal and opposite, so sum to zero, while the Neumann derivatives from the $G_{m+1}$ graphs are zero. It follows that $\Delta f(u)=\Delta f(v)=0=\lambda f(u)=\lambda f(v)$, so $f$ is an eigenfunction with the stated properties.
\end{proof}

With these two constructions in hand we choose to distinguish two sets of eigenvalues of $\Delta(\Gamma_n)$ according to the values of eigenfunctions at the gluing points $u,v$ in Figure~\ref{fig:4WayGluing} at the level of birth of the eigenvalue.  

\begin{definition}\label{def:02Series}
Consider $\lambda\in \spec(\Delta(\Gamma_n))$ that is born at level $n$.  If $n\geq2$ we say $\lambda$ is in the 2-series if it has an eigenfunction $f$ on $\Gamma_n$ with $f(u)=f(v)=0$ at the gluing points in Figure~\ref{fig:4WayGluing}.  We say $\lambda$ is in the 0-series at level $n$ if either $n=1$, or $n\geq 2$ and every corresponding $\Gamma_n$ eigenfunction $f$  has $f(u)f(v)\neq0$. We also put the eigenvalue $0$ corresponding to $\Gamma_0$ (which can be computed directly) in the 0-series; in this case there is only one vertex and the eigenfunction is non-zero there.
\end{definition}

We refer to~\cite{Bannonetal} for an explanation of the origin of this nomenclature, as it comes from the spectral dynamics of certain weighted Laplacians investigated by Grigorchuk and \.Zuk~\cite{GZ2} that are not otherwise used in the present work. 
However, we also note that in our later results it becomes apparent that the 2-series come from the $G_n$ graphs studied in~\cite{brzoska} whereas the 0-series reflect structure of the $\Gamma_n$ graphs and appear here for the first time.  Our next result tells us that  every $\lambda\in\spec(\Delta(\Gamma_n))$ is in either the 2-series or the 0-series.   

\begin{lemma}\label{lem:NoSingle0}
If $n\geq2$, $f$ is an eigenfunction of $\Delta(\Gamma_n)$ and $f$ vanishes at either of the gluing points $u$ or $v$ in Figure~\ref{fig:4WayGluing} then it vanishes at both $u$ and $v$.
\end{lemma}
\begin{proof}
Suppose the contrary and, without loss of generality due to the symmetry of Figure~\ref{fig:4WayGluing}, that $f(u)=0$, $f(v)=1$.  The corresponding eigenvalue $\lambda$ is born at some level $m$ and, since $\lambda\in\spec(\Delta(\Gamma_n))$, we have $m\leq n$. However, the restrictions of $f$ to the copies of $G_{n-1}$ and $G_n$ in Figure~\ref{fig:2WayGluing} must be eigenfunctions of $\Delta(G_{n-1})$ and $\Delta(G_n)$ respectively.  Since $f(v)=1$,  Theorem~\ref{thm:Gnefntheorem}~\eqref{item:Gnefnsimple} tells us that $m=n$ and the restriction of $f$ to the copy of $G_n$ is a non-zero multiple of $h_n$, so has non-zero equal Neumann derivatives at its boundary points.  But $\lambda$ being born at level $m=n$ says it is not in $\Dirspec(\Delta(G_{n-1}))$, so the restriction of $f$ to the copy of $G_{n-1}$ must vanish because its boundary data is zero from $f(u)=0$.

This reasoning tells us that $f(u)=0$ and $\Delta f(u)$ is the sum of the Neumann derivatives of the pieces glued at $u$. However, the Neumann derivatives from the copy of $G_{n-1}$ are both zero because $f$ vanishes there, while those from the copy $G_n$ are the Neumann derivatives of $h_n$ at the boundary points of $G_n$, which are non-zero and equal. It follows that $\Delta f(u)\neq0=\lambda f(u)$, so $f$ cannot be an eigenfunction.
\end{proof}

We establish some properties of the $0$-series and $2$-series eigenfunctions in the lemmas below.

\begin{lemma}\label{lem:0seriessimple}
If $\lambda\in\spec(\Delta(\Gamma_n))$ is in the $0$-series then it is simple and its eigenfunction either satisfies $f(u)=f(v)$ or $f(u)=-f(v)$.
\end{lemma}
\begin{proof}
If there are two distinct eigenfunctions with eigenvalue $\lambda$, each of which are non-zero at both gluing points, then there is a linear combination that vanishes at one of the gluing points and hence at both gluing points by Lemma~\ref{lem:NoSingle0}. Therefore $\lambda$ is from the 2-series. It follows that 0-series eigenvalues are simple.

Suppose  $f$ is an eigenfunction for a 0-series eigenvalue. Symmetrizing $f$ by reflection through the center symmetry line of Figure~\ref{fig:4WayGluing} cannot produce a different eigenfunction because there is only one, so the result must equal $f$ or be the zero function. In the former case $f$ is symmetric under the reflection and in the latter antisymmetric under the reflection.
\end{proof}

\begin{lemma}\label{lem:2seriesareDirevals}
$\lambda$ is in the 2-series and born at level $n$ if and only if $\lambda\in\Dirspec(\Delta(G_{n-2}))$ and is born at level $n-2$ for the latter spectrum.
\end{lemma}
\begin{proof}
Suppose $\lambda\in\spec(\Delta(\Gamma_n))$ is born at level $n$ in the 2-series and let $f$ denote the eigenfunction with $f(u)=f(v)=0$ in the notation of Figure~\ref{fig:4WayGluing}.  Then the restriction to the graphs $G_{n-1}$ and $G_{n-2}$ in this figure cannot all vanish and at least one is a Dirichlet eigenfunction.  It follows that $\lambda\in\Dirspec(\Delta(G_m))$ is born at level $m$ for some $m\leq n-1$. However Lemma~\ref{lem:Dirspectospec} gives $\lambda\in\spec(\Delta(\Gamma_{m+2}))$ and since $\lambda$ is born at level $n$ in this spectrum we have $m+2\geq n$.  So $m=n-1$ or $m=n-2$.

In order to arrive at a contradiction, suppose $m=n-1$ so the restrictions above are to $G_{n-1}=G_m$ and $G_{n-2}=G_{m-1}$ and both have boundary data $f(u)=f(v)=0$. Since $\lambda$ is born at level $m$ as a Dirichlet eigenvalue we see that the restriction to $G_{m-1}$ is identically zero and hence has zero Neumann derivatives at $u$ and $v$.  At the same time, Theorem~\ref{thm:Gnefntheorem}~\eqref{item:Gnefnsimple} says the restriction to each copy of $G_m$ is a multiple of a specific function $h_m$ which has equal and non-zero Neumann derivatives at its boundary points. At least one of these multiples must be non-zero or $f\equiv0$ and is not an eigenfunction.  Without loss of generality assume the non-zero multiple of $h_m$ is on the copy of $G_m$ where the boundary points are glued at $u$ in Figure~\ref{fig:4WayGluing}. Then the Neumann derivatives sum to give a non-zero contribution to the Laplacian at this point, and since the other contributions are zero because they are from the copies of $G_{m-1}$ where $f\equiv0$ we arrive at the contradiction $0\neq\Delta f(u)=\lambda f(u)=0$.

We have established that if $\lambda\in\spec(\Delta(\Gamma_n))$ is born at level $n$ in the 2-series then $\lambda\in\Dirspec(\Delta(G_{n-2}))$ and is born at level $n-2$.  For the converse, suppose $\lambda\in\Dirspec(\Delta(G_m))$ is born at level $m$. Then Lemma~\ref{lem:Dirspectospec} shows $\lambda\in\spec(\Delta(\Gamma_{m+2}))$ so is born at or before this level. However, if it is born at level $k<m+2$ then the previous argument shows $\lambda\in\Dirspec(\Delta(G_{k-2}))$ with $k-2<m$, contradicting that the level of birth in this spectrum is $m$.
\end{proof}

\begin{corollary}\label{lem:2SeriesSimpleBirth}
If $\lambda$ is in the 2-series and born at level $n$ it is a simple eigenvalue in $\spec(\Delta(\Gamma_n))$.
\end{corollary}
\begin{proof}
Let $f$ be an eigenfunction corresponding to $\lambda$. From Lemma~\ref{lem:2seriesareDirevals}, $\lambda\in\Dirspec(\Delta(G_{n-2}))$ and is born at level $n-2$. Applying Theorem~\ref{thm:Gnefntheorem}~\eqref{item:Gnefnn-modd} we see there is no Dirichlet eigenfunction on $G_{n-1}$ with eigenvalue $\lambda$, so the fact that $f(u)=f(v)=0$ in  Figure~\ref{fig:4WayGluing} ensures $f$ vanishes on the two copies of $G_{n-1}$ in that figure. Moreover, Theorem~\ref{thm:Gnefntheorem}~\eqref{item:Gnefnsimple} says the restriction to each copy of $G_{n-2}$ is a multiple of a specific function $h_{n-2}$ which has equal and non-zero Neumann derivatives at its boundary points. Since the eigenfunction equation says $\Delta f(u)=\lambda f(u)=0$ and the same at $v$, we discover that the Neumann derivatives from the two copies of $G_{n-2}$ must have equal and opposite signs. It follows that, up to a non-zero scalar multiple, $f$ is equal to $h_{n-2}$ on one copy of $G_{n-2}$ and $-h_{n-2}$ on the other copy, so the eigenspace is one-dimensional.
\end{proof}

\begin{lemma}\label{lem:2seriesvanishatuandv}
If $\lambda$ is in the $2$-series and born at level $m$ and $f$ is an eigenfunction with eigenvalue $\lambda$ on $\Gamma_n$ for some $n\geq m$ then decomposing $\Gamma_n$ as in Figure~\ref{fig:4WayGluing} we have $f(u)=f(v)=0$  
\end{lemma}
\begin{proof}
Suppose to the contrary that there is an eigenfunction $f$ on $\Gamma_n$ which does not vanish at the stated gluing points. Using symmetry or Lemma~\ref{lem:NoSingle0} we may assume without loss of generality that $f(u)\neq0$.

Compare the decomposition of $\Gamma_n$ into $G_n$ and $G_{n-1}$ in Figure~\ref{fig:2WayGluing} with that of $G_{n+1}$ into $G_n$ and two copies of $G_{n-1}$ in Figure~\ref{fig:G_nGluing}.  In both cases we identify two boundary points from $G_{n-1}$ and two from $G_n$.  It follows that restricting $f$ from $\Gamma_n$ to  its constituent $G_n$ and $G_{n-1}$ graphs, duplicating the latter, and then assembling these functions according to Figure~\ref{fig:G_nGluing} gives a solution of $\Delta f= \lambda f$ on $G_{n+1}\setminus\partial G_{n+1}$ with both boundary values equal to $f(u)$, where $u$ is the gluing point in Figure~\ref{fig:2WayGluing}, and also with $f(v)=f(u)$ where $v$ is as in Figure~\ref{fig:G_nGluing}.  (Note that here $v$ is not the point described in the statement of the lemma we are proving.)

Now we use the hypothesis that $\lambda$ is in the $2$-series and born at level $m$  and Lemma~\ref{lem:2seriesareDirevals} to see that $\lambda\in\Dirspec(\Delta(G_{m-2}))$.  We are then in the setting of Theorem~\ref{thm:Gnefntheorem}, with a solution of the eigenfunction equation at non-boundary points of $G_{n+1}\setminus\partial G_{n+1}$, the eigenvalue from the Dirichlet spectrum born at level $m-2$ and non-zero boundary data.  Since $n+1-(m-2)=n-m+3$ and $n\geq m$ we see that we are in one of two cases, both of which lead to a contradiction.  Specifically, if $n-m+3\geq3$ is odd we are in case~\eqref{item:notGnefn4} of the theorem, but here $f(v)=0$ and our $f(v)=f(u)\neq0$, a contradiction. Also, if $n-m+3$ is even then $n-m+3\geq4$ and we are in case~\eqref{item:notGnefn5}, where there are no functions with non-zero boundary data. 
\end{proof}

\begin{figure}
    \begin{tikzpicture}[scale=2]
        \draw[color=red] (-3,0) .. controls (-3.5,0.5) and (-3.5,-0.5) .. (-3,0);
        \draw[color=red] (-3,0) .. controls (-2.9,0.3) and (-2.6,0.3) .. (-2.5,0);
        \draw[color=red] (-3,0) .. controls (-2.9,-0.3) and (-2.6,-0.3) .. (-2.5,0);

        \draw (-2.5,0) .. controls (-2.5,0.3) and (-2.3,0.5) .. (-2,0.5);
        \draw (-2.5,0) .. controls (-2.5,-0.3) and (-2.3,-0.5) .. (-2,-0.5);
        \draw (-2,0.5) .. controls (-1.7,0.5) and (-1.5,0.3) .. (-1.5,0);
        \draw (-2,-0.5) .. controls (-1.7,-0.5) and (-1.5,-0.3) .. (-1.5,0);
        \draw (-2,0.5) .. controls (-2.5,1) and (-1.5,1) .. (-2,0.5);
        \draw (-2,-0.5) .. controls (-2.5,-1) and (-1.5,-1) .. (-2,-0.5);
        
        \draw[color=red] (-1.5,0) .. controls (-1.5,0.3) and (-1.3,0.5) .. (-1,0.5);
        \draw[color=red] (-1.5,0) .. controls (-1.5,-0.3) and (-1.3,-0.5) .. (-1,-0.5);
        \draw[color=red] (-1,0.5) .. controls (-0.7,0.5) and (-0.5,0.3) .. (-0.5,0);
        \draw[color=red] (-1,-0.5) .. controls (-0.7,-0.5) and (-0.5,-0.3) .. (-0.5,0);
        \draw[color=red] (-1,0.5) .. controls (-1.5,1) and (-0.5,1) .. (-1,0.5);
        \draw[color=red] (-1,-0.5) .. controls (-1.5,-1) and (-0.5,-1) .. (-1,-0.5);
        
        \draw (-0.5,0) .. controls (-0.5,0.2) and (-0.4,0.3) .. (-0.3,0.3);
        \draw (-0.5,0) .. controls (-0.5,-0.2) and (-0.4,-0.3) .. (-0.3,-0.3);
        \draw (0.5,0) .. controls (0.5,0.2) and (0.4,0.3) .. (0.3,0.3);
        \draw (0.5,0) .. controls (0.5,-0.2) and (0.4,-0.3) .. (0.3,-0.3);
        \draw (-0.3,0.3) .. controls (-0.3,0.4) and (-0.2,0.5) .. (0,0.5);
        \draw (-0.3,-0.3) .. controls (-0.3,-0.4) and (-0.2,-0.5) .. (0,-0.5);
        \draw (0.3,0.3) .. controls (0.3,0.4) and (0.2,0.5) .. (0,0.5);
        \draw (0.3,-0.3) .. controls (0.3,-0.4) and (0.2,-0.5) .. (0,-0.5);

        \draw (-0.3,0.3) .. controls (-0.9,0.3) and (-0.3,0.9) .. (-0.3,0.3);
        \draw (-0.3,-0.3) .. controls (-0.9,-0.3) and (-0.3,-0.9) .. (-0.3,-0.3);
        \draw (0.3,0.3) .. controls (0.9,0.3) and (0.3,0.9) .. (0.3,0.3);
        \draw (0.3,-0.3) .. controls (0.9,-0.3) and (0.3,-0.9) .. (0.3,-0.3);

        \draw[color=red] (0,0.5) .. controls (-0.2,0.6) and (-0.2,0.9) .. (0,1);
        \draw[color=red] (0,0.5) .. controls (0.2,0.6) and (0.2,0.9) .. (0,1);
        \draw[color=red] (0,1) .. controls (-0.5,1.5) and (0.5,1.5) .. (0,1);
        \draw[color=red] (0,-0.5) .. controls (-0.2,-0.6) and (-0.2,-0.9) .. (0,-1);
        \draw[color=red] (0,-0.5) .. controls (0.2,-0.6) and (0.2,-0.9) .. (0,-1);
        \draw[color=red] (0,-1) .. controls (-0.5,-1.5) and (0.5,-1.5) .. (0,-1);

        \draw[color=red] (1.5,0) .. controls (1.5,0.3) and (1.3,0.5) .. (1,0.5);
        \draw[color=red] (1.5,0) .. controls (1.5,-0.3) and (1.3,-0.5) .. (1,-0.5);
        \draw[color=red] (1,0.5) .. controls (0.7,0.5) and (0.5,0.3) .. (0.5,0);
        \draw[color=red] (1,-0.5) .. controls (0.7,-0.5) and (0.5,-0.3) .. (0.5,0);
        \draw[color=red] (1,0.5) .. controls (1.5,1) and (0.5,1) .. (1,0.5);
        \draw[color=red] (1,-0.5) .. controls (1.5,-1) and (0.5,-1) .. (1,-0.5); 

        \draw (2.5,0) .. controls (2.5,0.3) and (2.3,0.5) .. (2,0.5);
        \draw (2.5,0) .. controls (2.5,-0.3) and (2.3,-0.5) .. (2,-0.5);
        \draw (2,0.5) .. controls (1.7,0.5) and (1.5,0.3) .. (1.5,0);
        \draw (2,-0.5) .. controls (1.7,-0.5) and (1.5,-0.3) .. (1.5,0);
        \draw (2,0.5) .. controls (2.5,1) and (1.5,1) .. (2,0.5);
        \draw (2,-0.5) .. controls (2.5,-1) and (1.5,-1) .. (2,-0.5);

        \draw[color=red] (3,0) .. controls (3.5,0.5) and (3.5,-0.5) .. (3,0);
        \draw[color=red] (3,0) .. controls (2.9,0.3) and (2.6,0.3) .. (2.5,0);
        \draw[color=red] (3,0) .. controls (2.9,-0.3) and (2.6,-0.3) .. (2.5,0);

        \filldraw [red] (-3,0) circle (2pt);
        \filldraw [cyan] (-2.5,0) circle (2pt);
        \filldraw [black] (-2,-0.5) circle (2pt);
        \filldraw [black] (-2,0.5) circle (2pt);
        \filldraw [cyan] (-1.5,0) circle (2pt);
        \filldraw [red] (-1,0.5) circle (2pt);
        \filldraw [red] (-1,-0.5) circle (2pt);
        \filldraw [cyan] (-0.5,0) circle (2pt);
        \filldraw [black] (-0.3,0.3) circle (2pt);
        \filldraw [cyan] (0,0.5) circle (2pt);
        \filldraw [red] (0,1) circle (2pt);
        \filldraw [black] (0.3,0.3) circle (2pt);
        \filldraw [black] (-0.3,-0.3) circle (2pt);
        \filldraw [cyan] (0,-0.5) circle (2pt);
        \filldraw [red] (0,-1) circle (2pt);
        \filldraw [black] (0.3,-0.3) circle (2pt);
        \filldraw [cyan] (0.5,0) circle (2pt);
        \filldraw [cyan] (1.5,0) circle (2pt);
        \filldraw [red] (1,0.5) circle (2pt);
        \filldraw [red] (1,-0.5) circle (2pt);
        \filldraw [black] (2,0.5) circle (2pt);
        \filldraw [black] (2,-0.5) circle (2pt);
        \filldraw [cyan] (2.5,0) circle (2pt);
        \filldraw [red] (3,0) circle (2pt);

        \draw (-3.7,0) node {\(G_{m-1}\)};
        \draw (-2,1) node {\(G_{m-2}\)};
        \draw (-2,-1) node {\(G_{m-2}\)};
        \draw (-1,1) node {\(G_{m-1}\)};
        \draw (-1,-1) node {\(G_{m-1}\)};
        \draw (-0.45,0.75) node {\(G_{m-2}\)};
        \draw (-0.45,-0.75) node {\(G_{m-2}\)};
        \draw (0.45,0.75) node {\(G_{m-2}\)};
        \draw (0.45,-0.75) node {\(G_{m-2}\)};
        \draw (0,1.5) node {\(G_{m-1}\)};
        \draw (0,-1.5) node {\(G_{m-1}\)};
        \draw (1,1) node {\(G_{m-1}\)};
        \draw (1,-1) node {\(G_{m-1}\)};
        \draw (2,1) node {\(G_{m-2}\)};
        \draw (2,-1) node {\(G_{m-2}\)};
        \draw (3.7,0) node {\(G_{m-1}\)};
    \end{tikzpicture}
    \caption{\(\Gamma_{m+2}\) as a gluing of copies of $G_{m-1}$ (in red) and of $G_{m-2}$ (in black) with $(m-1,m-2)$-gluing points in blue.}
    \label{fig:Gk+2Deconstructed}
\end{figure}
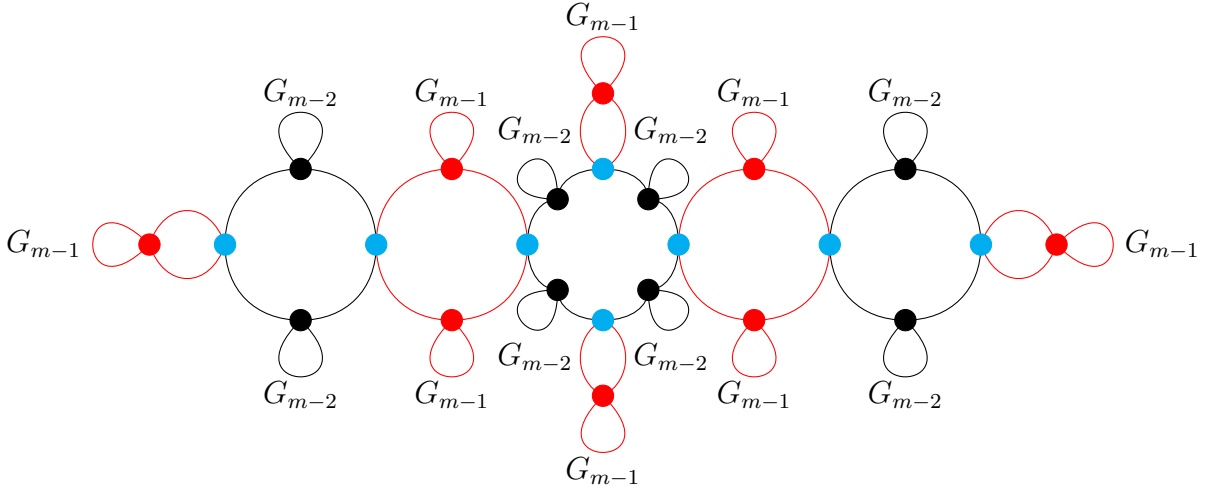

In what follows we determine a basis for the $2$-series eigenfunctions on $\Gamma_n$ that are born at level $m\leq n-2$.  To do so, we first decompose $\Gamma_n$ into two copies of each of $G_{n-1}$ and $G_{n-2}$ as in Figure~\ref{fig:4WayGluing}, in which we note that $u$ and $v$ are each obtained by identifying two boundary points of copies of $G_{n-1}$ and two boundary points of copies of $G_{n-2}$. Let us call these $(n-1,n-2)$-gluing points.  If $m=n-2$ we stop; otherwise we apply the decomposition of Figure~\ref{fig:G_nGluing} to decompose $G_{n-1}$ into $G_{n-2}$ and two copies of $G_{n-3}$.  The result is a graph consisting of copies of $G_{n-2}$ and $G_{n-3}$, with the vertices common to both being $(n-2,n-3)$-gluing points.  We may repeat recursively for $k=1,\dotsc,n-m$, at each stage decomposing $G_{n-k}$ into $G_{n-k-1}$ and two copies of $G_{n-k-2}$,  so as to obtain a graph made of copies of $G_{m-1}$ and $G_{m-2}$, meeting at $(m-1,m-2)$-gluing points.  Figure~\ref{fig:Gk+2Deconstructed} shows this decomposition with $n=m+2$.

\begin{lemma}\label{lem:vanishatgluingofGkGk+1}
In the preceding decomposition, any $2$-series eigenfunction on $\Gamma_n$ born at level $m\leq n$ vanishes at the $(m-1,m-2)$-gluing points. It therefore vanishes identically on the copies of $G_{m-1}$ in the decomposition and its restriction to the level $m-2$ vertices is a scalar multiple of the function $h_{m-2}$ from Theorem~\ref{thm:Gnefntheorem}~\eqref{item:Gnefnsimple} on each of the copies of $G_{m-2}$ in the decomposition.
\end{lemma}
\begin{proof}
We induct on the index of the gluing points, with the base case of the $(n-1,n-2)$-gluing points (i.e. $u$ and $v$ from Figure~\ref{fig:4WayGluing}) being Lemma~\ref{lem:2seriesvanishatuandv}. Suppose the eigenfunction $f$ vanishes at the $(n-k,n-k-1)$-gluing points for some $1\leq k\leq n-m$ and consider the restriction of $f$ to any copy of $G_{n-k}$ in the decomposition described above. It will be convenient to denote the gluing point (from Figure~\ref{fig:G_nGluing}) of $G_{n-k}$ by $v'$.  Because $f$ vanishes at the $(n-k,n-k-1)$-gluing points, its restriction to $G_{n-k}$ is a Dirichlet eigenfunction with eigenvalue that was born for the $2$-series at level $m$ and thus born for the Dirichlet spectrum on the $G_j$ graphs at level $m'=m-2$ by Lemma~\ref{lem:2seriesareDirevals}. These functions are described in Theorem~\ref{thm:Gnefntheorem}.  Specifically, we have a Dirichlet eigenfunction on $G_{n'}$ with $n'=n-k$, it has eigenvalue born for the Dirichlet spectrum of $G_{m'}$ with $m'=m-2$ and our condition on $k$ gives $m'\leq n'-2$, so the relevant functions are in~\eqref{item:Gnefnn-modd}--\eqref{item:DNefns} of Theorem~\ref{thm:Gnefntheorem}.  The important consideration for the present argument is that all functions described in those parts of the theorem vanish at the gluing point $v'$.  It follows that when we decompose $G_{n-k}$ into a copy of $G_{n-k-1}$ and two of $G_{n-k-2}$ following Figure~\ref{fig:G_nGluing} the function $f$ vanishes at the resulting $(n-k-1,n-k-2)$-gluing point, closing the induction.

Since $f$ vanishes at the $(m-1,m-2)$-gluing points its restriction to each copy of $G_{m-1}$ in the decomposition has Dirichlet boundary data and satisfies the eigenfunction equation for an eigenvalue born in $\Dirspec(\Delta(G_{m-2}))$. Theorem~\ref{thm:Gnefntheorem}~\eqref{item:Gnefnn-modd} says there are no eigenfunctions like this, so $f$ must vanish identically on these subgraphs.  Similarly, the restriction of $f$ to each copy of $G_{m-2}$ in the decomposition is a Dirichlet eigenfunction, and since  Theorem~\ref{thm:Gnefntheorem}~\eqref{item:Gnefnsimple}  says there is a one-dimensional space of these with basis $h_{m-2}$ the proof is complete.
\end{proof}

A description of a basis for the $2$-series eigenfunctions on $\Gamma_n$ born at level $m\leq n$ follows without difficulty.  Notice that the eigenfunctions in this basis are obtained by connecting the basis elements from Theorem~\ref{thm:Gnefntheorem}~\eqref{item:DnotNefn} into loops, so these basis elements reflect the topological structure of the graph. 

\begin{theorem}\label{thm:2seriesefns}
Fix $\lambda$ a $2$-series eigenvalue born at level $m\leq n$ and recall from  Theorem~\ref{thm:Gnefntheorem}~\eqref{item:Gnefnsimple} that there is a unique Dirichlet eigenfunction $h_{m-2}$ on $G_{m-2}$ with eigenvalue $\lambda$.  Decompose $\Gamma_n$ into copies of $G_{m-1}$ and $G_{m-2}$ and consider a loop in $\Gamma_n$ consisting of an even number of copies of $G_{m-2}$, each copy having one boundary point identified with the next copy. Construct a function $f$ on $\Gamma_n$ by taking it equal to $\pm h_{m-2}$ on each copy of $G_{m-2}$ in the loop, with the signs arranged in an alternating manner, and setting it to be zero on the rest of $\Gamma_n$.  This is a $2$-series eigenfunction on $\Gamma_n$ with eigenvalue $\lambda$. Moreover, the collection of such eigenfunctions is a basis for the $\lambda$-eigenspace on $\Gamma_n$, so the dimension of the eigenspace is the number of loops of this type.
\end{theorem}
\begin{proof}
We first check the construction of a $\lambda$-eigenfunction on the loop.
Observe that arranging the copies of $h_{m-2}$ with alternating signs is possible because the number of copies of $G_{m-2}$ in the loop is even.  Since $h_{m-2}$ satisfies $\Delta f=\lambda f$  on $G_{m-2}\setminus\partial G_{m-2}$ and has Dirichlet boundary data, the construction clearly produces a function satisfying $\Delta f=\lambda f$ at all points except perhaps the $(m-1,m-2)$-gluing points on the loop. A gluing point $x$ of this type is obtained by identifying two boundary points from copies of $G_{m-1}$ and two boundary points from copies of $G_{m-2}$.  The two copies of $G_{m-2}$ are necessarily distinct and lie on the loop, but the two copies of $G_{m-1}$ could coincide (actually being one copy with both boundary points glued at $x$). 
Since $f\equiv0$ on the copy or copies of $G_{m-1}$ these make no contribution to the Laplacian at $x$. Since one copy of $G_{m-2}$  carries the function $h_{m-2}$ and the other carries $-h_{m-2}$, and the Neumann derivatives of $h_{m-2}$ are equal at both boundary points (see Theorem~\ref{thm:Gnefntheorem}~\eqref{item:Gnefnsimple}), we conclude that the sum of Neumann derivatives from the two copies of $G_{m-2}$ is zero, so $\Delta f(x)=0$. It is also the case that $f(x)=0$ because $h_{m-2}$ has Dirichlet boundary conditions and $f$ vanishes on the copies of $G_{m-1}$. Hence $\Delta f(x)=\lambda f(x)$ at all $(m-1,m-2)$-gluing points and  $f$ is an eigenfunction with eigenvalue $\lambda$. Since this set of gluing points includes $u$ and $v$ as in Figure~\ref{fig:4WayGluing} it is a $2$-series eigenfunction.

Now we verify the eigenfunctions corresponding to loops span the $\lambda$ eigenspace.  Suppose $f$ is a $2$-series eigenfunction born at level $m$. By Lemma~\ref{lem:vanishatgluingofGkGk+1} it must consist of scalar multiples of $h_{m-2}$ on copies of $G_{m-2}$ and be zero on the copies of $G_{m-1}$.  Multiplying $f$ by a non-zero constant if necessary, we may take a copy of $G_{m-2}$ on which $f=h_{m-2}$. At a boundary point $x$ of this copy of $G_{m-2}$ we have $f(x)=0$ and $f$ has a non-zero Neumann derivative. The copy or copies of $G_{m-1}$ meeting at $x$ make no contribution to the Laplacian, so in order to have $\Delta f(x)=\lambda f(x)=0$ we must have $f=-h_{m-2}$ on the (unique) other copy of $G_{m-2}$ that meets $x$. Repeating this argument we uniquely construct a path of copies of $G_{m-2}$ on which there is an alternating sequence of copies of $\pm h_{m-2}$, and the only way it can terminate producing an eigenfunction is if it reaches the second boundary point of our original copy of $G_{m-2}$ such that the resulting loop contains an even number of copies of $G_{m-2}$.  Reasoning similarly for each copy of $G_{m-2}$ we conclude that the restriction of $f$ to any even loop of copies of $G_{m-2}$ is a multiple of one of the eigenfunctions described in the statement of the theorem, and  that its restriction to any other copy of $G_{m-2}$ must be zero.  Hence it is in the span of the eigenfunctions we have described.

Finally, observe that the preceding argument shows any copy of $G_{m-2}$ can belong to at most one loop of copies of $G_{m-2}$, so the loops do not intersect and the eigenfunctions we have described are linearly independent.  We conclude that they form a basis for the $\lambda$ eigenspace on $\Gamma_n$.
\end{proof}

\begin{lemma}\label{lem:2SeriesMultiplicity}
If $\lambda \in \spec(\Delta(\Gamma_n))$ is in the 2-series and is born at level $m\leq n$ on $\Gamma_m$, then the multiplicity of $\lambda$ at level $n$ is given by $\sigma_{n-m}$ where $\sigma_k=\frac16 (2^{2+k}+3+(-1)^{k+1})$.
\end{lemma}
\begin{proof}
From Theorem~\ref{thm:2seriesefns} it is sufficient to count the number $\sigma_{n-m}$ of even loops of copies of $G_{m-2}$ obtained when we decompose $\Gamma_n$ into copies of $G_{m-1}$ and $G_{m-2}$.  We do so by induction on $n$ with base case $n=m$, where Figure~\ref{fig:4WayGluing} makes it clear the number of loops is $1=\sigma_0$.

Suppose $n>m$.  Consider the decomposition of $\Gamma_n$ into a copy of $G_n$ and one of $G_{n-1}$ in Figure~\ref{fig:2WayGluing}. Since these copies meet at a point, any loop of copies of $G_{m-2}$ in $\Gamma_n$ must lie entirely within the copy of $G_n$  or entirely within the copy of $G_{n-1}$.  Denote the first as a $\Gamma_n$ loop of type $1$ and the second as a $\Gamma_n$ loop of type $2$.  Note that there is at most one $\Gamma_n$ loop through the gluing point $u$ in the figure because loops do not intersect. It is easy to check by recursing the decomposition in Figure~\ref{fig:G_nGluing} that this loop is in $G_n$ (hence is type $1$) if $n-m$ is even and in $G_{n-1}$ (so is type $2$) if $n-m$ is odd.

Now look at the decomposition of $\Gamma_{n+1}$ into two copies of each of $G_n$ and $G_{n-1}$ using Figure~\ref{fig:4WayGluing}. Since the copies of $G_n$ have both boundary points identified, just as they were in $\Gamma_n$ in Figure~\ref{fig:2WayGluing}, we conclude immediately that for any $\Gamma_n$ loop of type $1$ there are two $\Gamma_{n+1}$ loops.  The same can be said for any $\Gamma_n$ loops of type $2$ through the gluing point $u$ in  Figure~\ref{fig:2WayGluing}.  If $n-m$ is even we saw above that there are no type $2$ loops through this point, so in this case $\sigma_{n+1-m}=2\sigma_{n-m}$.  If $n-m$ is odd then there is a $\Gamma_n$ loop of type $2$ through the identified boundary points. This does not give two loops on $\Gamma_{n+1}$, but instead we see that there is a unique loop of copies of $G_{m-2}$ through both $u$ and $v$ and supported on the copies of $G_{n-1}$. It follows that $\sigma_{n+1-m}= 2\sigma_{n-m}-1$ in this case.

It is easy to verify that the stated formula $\sigma_k$ satisfies
\begin{equation*}
	\sigma_{k+1} = \begin{cases} 2\sigma_{k} &\text{ if $k$ is even}\\
				2\sigma_k -1 &\text{ if $k$ is odd}
				\end{cases}
	\end{equation*}
so the proof is complete.
\end{proof}

\section{Dynamics for characteristic polynomials of the Laplacians.}\label{sec:dynamics}

It was recognized by Grigorchuk and \.Zuk~\cite{GZ2} that the spectra of certain weighted Laplacian operators on the graphs $\Gamma_n$, in which certain edges have weight $\mu$ and the others weight $1$, are related by a two-dimensional dynamical system which operates on the weight and the eigenvalue. 
The subsequent work in~\cite{brzoska} found a different dynamical description of the Dirichlet spectra of the operators $\Delta(G_n)$ in the situation where all edges have equal weight.  In this latter approach one first decomposes the spectrum according to the topology of the graphs, as in Section~\ref{sec:topologyofspectrum}, counts the multiplicities of the eigenvalues, and then determines a dynamical system that relates certain factors in the characteristic polynomials that correspond to specific eigenspaces. In this section we generalize the latter method to the graphs $\Gamma_n$.  Our main result, Theorem~\ref{thm:mainrecursion} is a generalization of the recursion in Theorem~\ref{thm:brzrecursion}, which is a restatement of Corollary~3.16 of~\cite{brzoska}.

\subsection*{Factorization of characteristic polynomials via topology of the graphs}

\begin{definition}\label{def:GammaPoly}
Following the notation in~\cite{brzoska} we let $\gamma_n$ be the monic polynomial whose roots are the eigenvalues born at level $n$ for the Dirichlet Laplacian on $G_n$.  Also define $\psi_n$ to be the monic polynomial whose roots are the $0$-series eigenvalues born at level $n$. To simplify some recursions we also set $\gamma_0=\gamma_{-1}=\gamma_{-2}=1$, which is consistent with the preceding definition because $G_0$ has no Dirichlet eigenvalues.
\end{definition}

\begin{lemma}
$\gamma_{n-2}$ is the monic polynomial with roots the $2$-series eigenvalues born at level $n$ on $\Gamma_n$.
\end{lemma}
\begin{proof}
From Lemma~\ref{lem:2seriesareDirevals} the sets of eigenvalues are the same. From Corollary~\ref{lem:2SeriesSimpleBirth} the eigenvalues are simple for $\Delta(\Gamma_n)$ and from Theorem~\ref{thm:Gnefntheorem}~\eqref{item:Gnefnsimple} they are simple as Dirichlet eigenvalues for $\Delta(G_{n-2})$.
\end{proof}

\begin{theorem}\label{thm:FullFactorization}
The characteristic polynomial of the spectrum of $\Gamma_n$ is
\begin{equation*}
P_n=\prod_{k=0}^{n} \psi_k\gamma_{k-2}^{\sigma_{n-k}}
\end{equation*}
where $\sigma_k= \frac16(2^{k+2}+3 + (-1)^{k+1})$.
\end{theorem}
\begin{proof}
All eigenvalues of $\Delta(\Gamma_n)$ are in the 0-series or 2-series by Lemma~\ref{lem:NoSingle0}, so $P_n$ is the product of $\gamma_k$ and $\psi_k$ factors raised to the correct multiplicities.  All $0$-series eigenvalues are simple by Lemma~\ref{lem:0seriessimple}.  The $2$-series eigenvalues born at level $k$ are simple when born, so are multiplicity $1$ roots of $\gamma_{k-2}$. They occur with multiplicity $\sigma_{n-k}$ on $\Gamma_n$ by Lemma~\ref{lem:2SeriesMultiplicity}, so give factors $\gamma_{k-2}^{\sigma_{n-k}}$ for $k\geq3$, but for $k=0,1,2$ we have defined $\gamma_{k-2}=1$.
\end{proof}

It will be helpful later to compare this to the formula for the characteristic polynomial of the Dirichlet Laplacian matrix on $G_n$.  In the notation of~\cite{brzoska} this is called $c_n$, and the following result is established.
\begin{theorem}[\protect{Theorem~3.13 of~\cite{brzoska}}]\label{thm:Brzfmlaforcn}
\begin{gather*}
c_n=\gamma_{n}\prod_{k=1}^{n-1} \gamma_k^{S_{n-k}}, \text{ where}\\
S_n=\frac1{36}\bigl( 9+23(-1)^n+2^{2+n}-6n(-1)^n\bigr).
\end{gather*}
\end{theorem}

The following connection between the $P_n$ and $c_n$ may be established either by recognizing that the Dirichlet but not Neumann terms in $c_n$ are from $\gamma_n$ while the remaining factors correspond to the Dirichlet-Neumann eigenfunctions, or by the direct computation we give below.

\begin{proposition}\label{prop:Pnusingcn}
For $n\geq2$,
\begin{equation*}
P_n = \frac{c_nc_{n-1}}{\gamma_n\gamma_{n-1}} \prod_{k=0}^n\psi_k .
\end{equation*}
\end{proposition}
\begin{proof}
Equality for the $\psi_k$ factors is obvious. If we first check that
\begin{equation*}
S_m+S_{m-1}=\frac1{36}\bigl( 18+(4+2)2^{m} +(-1)^{m}(6(m-1)-m)\bigr)=\frac16(3+2^m-(-1)^m)=\sigma_{m-2}
\end{equation*}
then the statement follows from computing
\begin{equation*}
	  \prod_1^{n} \gamma_{k-2}^{\sigma_{n-k}}
	  = \prod_1^{n-2} \gamma_k^{\sigma_{n-k-2}}
	  = \prod_1^{n-2} \gamma_k^{S_{n-k}+S_{n-k-1}}
	  =   \biggl( \prod_1^{n-1}\gamma_k^{S_{n-k}}\biggr)\biggl( \prod_1^{n-2} \gamma_k^{S_{n-k-1}}\biggr) 
	  \end{equation*}
in which we used $\gamma_0=\gamma_{-1}=1$ for the first equality, $\sigma_{n-k-2}=S_{n-k}+S_{n-k-1}$ in the second and $S_1=0$ in the third. The requirement $n\geq2$ ensures all subscripts are within the defined ranges.
\end{proof}

\begin{corollary}\label{cor:deg(psi_n)}
We have $\deg(\psi_0)=1$, $\deg(\psi_1)=2$ and for $n\geq2$
 \begin{align*}
 \deg\left(\psi_n\right) &=\deg\left(\gamma_n\right)-\deg\left(\gamma_{n-2}\right)
 =\frac2{\sqrt{7}} \sum_{j=1,2,3} (\rho_j^n-\rho_j^{n-2})\cos\bigl(\phi+\frac{2\pi j}3\bigr)
 	\end{align*}
where $\phi=\frac13\arctan(-3\sqrt3)$ and the $\rho_j$ are the roots of $\rho^3-\rho^2-2\rho+1$, specifically
 \begin{equation*}
 \rho_1=\frac13\Bigl(1-2\sqrt{7}\cos\phi\Bigr),\quad  \rho_2=\frac13\Bigl(1-2\sqrt{7}\cos\bigl(\phi+\frac{2\pi}3\bigr) \Bigr),\quad
  \rho_3=\frac13\Bigl(1+2\sqrt{7}\cos\bigl(\phi+\frac\pi3\bigr) \Bigr).
 \end{equation*}
\end{corollary}

\begin{remark}\label{rem:dynamicaldegree}
The root with largest magnitude, which is $|\rho_3|=\rho_3$, is actually the dynamical degree of the dynamics given in Theorem~\ref{thm:brzrecursion}, though we do not discuss this here.
\end{remark}

\begin{proof}
Observe from Proposition~\ref{prop:Pnusingcn} that 
\begin{equation*}
	\psi_n= \frac{P_n}{P_{n-1}} \frac{\gamma_nc_{n-2}}{c_n\gamma_{n-2}}
	\end{equation*}
and therefore
\begin{equation}\label{eqn:Pnusingcn1}
	\deg(\psi_n)=\deg(P_n)-\deg(P_{n-1})+\deg(c_{n-2})-\deg(c_n)+\deg(\gamma_n)-\deg(\gamma_{n-2}).
\end{equation}
We have $\deg(P_n)=2^n$ because it is the number of points in $\Gamma_n$, see Lemma~\ref{lem:SizeofGamma}. Similarly, $\deg(c_n)$ is the number of points in $C_n$, which is two less than the number of points in $G_n$ and is therefore $\frac16(2^{n+2}+(-1)^{n+1}-3)$ by Lemma~2.5 of~\cite{brzoska}. Substituting into~\eqref{eqn:Pnusingcn1} gives the first expression for $\deg(\psi_n)$ and the second is obtained by substituting the degrees of $\gamma_n$ and $\gamma_{n-2}$ from 
Proposition~3.17 of~\cite{brzoska}.
\end{proof}

\subsection*{Recursions for characteristic polynomials}

We derive a recursive formula for the characteristic polynomials of $\Delta(\Gamma_n)$. This method follows the approach in \cite{brzoska}, which utilizes the self-similarity of the graphs by successively removing vertices and analyzing how the determinant behaves under such operations. Unlike the approach of~\cite{GZ2}, which uses the Schur complement to deduce the dynamics, we argue directly from the following elementary lemma about the effect of deleting a point from a graph on its characteristic polynomial.

Let $L$ be the Laplacian matrix of a graph and denote by $D(L)(\lambda) = \det(\lambda I - L)$ its characteristic polynomial. We will use the following result, proved in \cite{brzoska}, in which the notation $L^A$ refers to the Laplacian matrix in which the rows and columns corresponding to vertices in the set $A$ are deleted.

\begin{lemma}[Lemma 3.1 in \cite{brzoska}]\label{lem:DeterminantByRemovingVertices}
Let $G$ be a finite graph, $u$ a fixed vertex, and $C(u)$ the set of simple cycles in $G$ containing $u$. Suppose $L$ is a matrix indexed by the vertices of $G$ with diagonal entries $d_j$  and off-diagonal entries $-c_{jk}$ such that $c_{jk}=0$ unless the $j$ and $k$ vertices of the graph are connected by an edge. If $D(\cdot)$ denotes the operation of taking the characteristic polynomial then
\begin{equation*}\label{eqn:LaplacianfromBGJRT}
D(L)(\lambda) = (\lambda -d_u) D(L^{\{u\}})(\lambda) - \sum_{v\sim u} c^2_{uv}D(L^{\{u,v\}})(\lambda) + 2\sum_{Z \in C(u)} (-1)^{n(Z)-1}\pi(Z)  D(L^Z)(\lambda),
\end{equation*}
where $n(Z)$ is the number of vertices in $Z$ and $\pi(Z)$ is the product of the edge weights $c_{jk}$ along $Z$.
\end{lemma}

We apply this identity to the vertex $u$ in the graph $\Gamma_n$ when it is viewed as in Figure~\ref{fig:2WayGluing}. It is apparent that $L^{\{u\}}$ will then involve the graphs $G_n$, $G_{n-1}$ and some other graphs derived from them by deleting vertices. Following the notation of~\cite{brzoska} we label these $B_n$, $C_n$, $D_n$, where $B_n$ is $G_n$ with one boundary point deleted, $C_n$ is $G_n$ with both boundary points deleted, and $D_n$ is $G_n$ with both boundary points deleted and also one vertex neighboring a boundary point deleted. The graphs $B_3$, $C_3$, and $D_3$ are shown in Figure \ref{fig:Level3Graphs}.  It will be convenient to write $b_n(\lambda)$, $c_n(\lambda)$, $d_n(\lambda)$ for the characteristic polynomials of $B_n$, $C_n$ and $D_n$, respectively.  Note that the roots of $c_n(\lambda)$ are the eigenvalues of the Dirichlet Laplacian on $G_n$.

\begin{figure}[h]
\begin{tikzpicture}[scale=0.6]
\draw  (4.5,3) node[circle,fill=lightgray,inner sep=2pt]{} -- (6,3) node[circle,fill,inner sep=2pt](300a){} --  (7.5,3) node[circle,fill,inner sep=2pt](3u){} --  (9,3) node[circle,fill,inner sep=2pt](300b){} -- (10.5,3) node[circle,fill,inner sep=2pt]{};
\path[-, every loop/.style= {looseness=10, distance=20, in=300,out=240}]   (3u)  --   (7.5,2) node[circle,fill,inner sep=2pt](311){} -- (7.5,1) node[circle,fill,inner sep=2pt](3101){} edge [loop below](7.5,0)  ;
\path[every loop/.style= {looseness=10, distance=20, in=300,out=240}]   (300a) edge [loop below]  () (300b) edge [loop below]();
\draw (3u) edge[bend right]  (311) edge[ bend left]  (311);
\draw (311) edge[bend right]  (3101) edge[ bend left]  (3101);
\node at (7.5,-.5){$B_3$};
\end{tikzpicture}\hspace{0.5in}
\begin{tikzpicture}[scale=0.6]
\draw  (4.5,3) node[circle,fill=lightgray,inner sep=2pt]{} -- (6,3) node[circle,fill,inner sep=2pt](300a){} --  (7.5,3) node[circle,fill,inner sep=2pt](3u){} --  (9,3) node[circle,fill,inner sep=2pt](300b){} -- (10.5,3) node[circle,fill=lightgray,inner sep=2pt]{};
\path[-, every loop/.style= {looseness=10, distance=20, in=300,out=240}]   (3u)  --   (7.5,2) node[circle,fill,inner sep=2pt](311){} -- (7.5,1) node[circle,fill,inner sep=2pt](3101){} edge [loop below](7.5,0)  ;
\path[every loop/.style= {looseness=10, distance=20, in=300,out=240}]   (300a) edge [loop below]  () (300b) edge [loop below]();
\draw (3u) edge[bend right]  (311) edge[ bend left]  (311);
\draw (311) edge[bend right]  (3101) edge[ bend left]  (3101);
\node at (7.5,-.5){$C_3$};
\end{tikzpicture}\hspace{0.5in}
\begin{tikzpicture}[scale=0.6]
\draw  (4.5,3) node[circle,fill=lightgray,inner sep=2pt]{} -- (6,3) node[circle,fill,inner sep=2pt](300a){} --  (7.5,3) node[circle,fill,inner sep=2pt](3u){} --  (9,3) node[circle,fill=lightgray,inner sep=2pt](300b){} -- (10.5,3) node[circle,fill=lightgray,inner sep=2pt]{};
\path[-, every loop/.style= {looseness=10, distance=20, in=300,out=240}]   (3u)  --   (7.5,2) node[circle,fill,inner sep=2pt](311){} -- (7.5,1) node[circle,fill,inner sep=2pt](3101){} edge [loop below](7.5,0)  ;
\path[every loop/.style= {looseness=10, distance=20, in=300,out=240}]   (300a) edge [loop below]  () (300b) edge [loop below]();
\draw (3u) edge[bend right]  (311) edge[ bend left]  (311);
\draw (311) edge[bend right]  (3101) edge[ bend left]  (3101);
\node at (7.5,-.5){$D_3$};
\end{tikzpicture}

\vspace{5mm}
    \begin{tikzpicture}[scale=1.5]
        \draw (-2.5,0) .. controls (-3,-0.5) and (-3,0.5) .. (-2.5,0);
        \draw (-2.5,0) .. controls (-2.25,0.25) and (-1.75,0.25) .. (-1.5,0);
        \draw (-2.5,0) .. controls (-2.25,-0.25) and (-1.75,-0.25) .. (-1.5,0);
        \draw (-1.5,0) .. controls (-1.25,0.25) and (-0.75,0.25) .. (-0.5,0);
        \draw (-1.5,0) .. controls (-1.25,-0.25) and (-0.75,-0.25) .. (-0.5,0);
        \draw (-0.5,0) .. controls (-0.45,0.4) and (-0.2,0.49) .. (0,0.5);
        \draw (-0.5,0) .. controls (-0.45,-0.4) and (-0.2,-0.49) .. (0,-0.5);
        \draw (0,0.5) .. controls (0.2,0.49) and (0.45,0.4) .. (0.5,0);
        \draw (0,-0.5) .. controls (0.2,-0.49) and (0.45,-0.4) .. (0.5,0);
        \draw (0,0.5) .. controls (-0.25,0.9) and (0.25,0.9) .. (0,0.5);
        \draw (0,-0.5) .. controls (-0.25,-0.9) and (0.25,-0.9) .. (0,-0.5);
        \draw (0.5,0) .. controls (0.75,0.25) and (1.25,0.25) .. (1.5,0);
        \draw (0.5,0) .. controls (0.75,-0.25) and (1.25,-0.25) .. (1.5,0);
        \draw (1.5,0) .. controls (1.75,0.25) and (2.25,0.25) .. (2.5,0);
        \draw (1.5,0) .. controls (1.75,-0.25) and (2.25,-0.25) .. (2.5,0);
        \draw (2.5,0) .. controls (3,-0.5) and (3,0.5) .. (2.5,0);
        
        \filldraw [black] (-2.5,0) circle (2pt);
        \filldraw [black] (-1.5,0) circle (2pt);
        \filldraw [gray] (-0.5,0) circle (2pt);
        \node at (-0.3, 0) {$u$};
        \filldraw [black] (0,-0.5) circle (2pt);
        \filldraw [black] (0,0.5) circle (2pt);
        \filldraw [black] (0.5,0) circle (2pt);
        \filldraw [black] (1.5,0) circle (2pt);
        \filldraw [black] (2.5,0) circle (2pt);

        \node at (0,-1.2){$\Gamma_3$};
    \end{tikzpicture}
 
    \caption{Graphs $B_3$, $C_3$, $D_3$, and $\Gamma_3$ with one vertex removed. Shaded vertices are removed in the corresponding matrices.}
    \label{fig:Level3Graphs}
   
\end{figure}
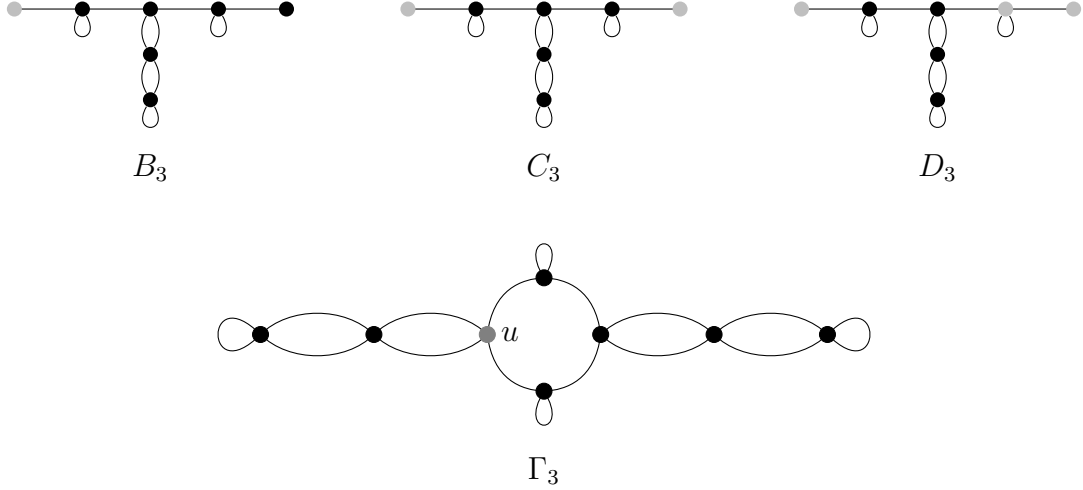

\begin{proposition}\label{prop:CharPolyfromBGJRT}
For $n\geq1$ the characteristic polynomial $P_n$ of  $\Gamma_n$ satisfies
\begin{equation}\label{eq:CharPoly1}
P_n = (\lambda - 4)c_n c_{n-1} - 2d_{n-1}c_n - 2d_n c_{n-1} - 2c_n g_{n-1} - 2c_{n-1}g_n,
\end{equation}
where 
\begin{equation}\label{eqn:DefOfgn}
	g_{n-1}=\prod_{1\leq j <\frac n2}\bigl( c_{n-2j}\bigr)^{2^{j-1}}.
\end{equation}
and $c_n$, $d_n$ are the characteristic polynomials of $C_n$ and $D_n$ respectively.
\end{proposition}

\begin{proof}
We apply Lemma~\ref{lem:DeterminantByRemovingVertices} to $u$ in Figure~\ref{fig:2WayGluing} using the fact that all vertices have degree $4$. When $u$ itself is deleted we obtain a disjoint union of one copy of $C_n$ and one copy of $C_{n-1}$, giving the term $(\lambda-4)c_nc_{n-1}$. When $u$ and a neighbor $v$ are deleted, we see that if $v$ is in the copy of $C_n$, then the result is a disjoint union of $D_n$ and $C_{n-1}$, while if $v$ is on the copy of $C_{n-1}$, then the result is a disjoint union of $D_{n-1}$ and $C_n$. Since there are two such neighbors on each copy, these terms appear with a factor of $2$. Finally, there are two simple cycles each with an even number of vertices (so $(-1)^{n(Z)-1}=-1$ in the Lemma), one in the copy of $G_{n-1}$ and one in the copy of $G_n$. When we delete the vertices on the simple cycle in the copy of $G_{n-1}$ we see that this graph decomposes into a disjoint union of one copy of $G_{n-2}$, two of $G_{n-4}$, and (inductively)  $2^{j-1}$ copies of $C_{n-2j}$ for each $j$ such that $2j < n$; there are also loops along this path which now have no vertices and therefore each have characteristic polynomial 1.  The product of the corresponding characteristic polynomials is $g_{n-1}$, and since deleting this cycle removes $u$ from the copy of $G_n$ we also get a factor $c_n$, finally leaving the term $-2c_ng_{n-1}$. Deleting the cycle in the copy of $G_n$ gives the term $-2c_{n-1}g_n$ in a similar manner, and adding the terms gives~\eqref{eq:CharPoly1}
\end{proof}

The same procedure used in  Proposition~\ref{prop:CharPolyfromBGJRT} was applied to the graphs $G_n$ and $C_n$ in~\cite{brzoska}. This led to recursions among the characteristic polynomials for the Laplacians on a somewhat larger family of graphs, from which various results  we will need later were established.  We summarize what we need in the following result.
\begin{lemma}[\protect{\cite{brzoska}}]
The following recursions and~\eqref{eqn:DefOfgn} can be used to determine all $b_n$, $c_n$, $d_n$ and $g_n$ from the initial data $c_0=1$, $c_1=\lambda-2$, $c_2=\lambda^2-6\lambda+4$.
\begin{gather}
b_n = (\lambda - 1)c_n - d_n \label{eqn:brzeqnfordnelim}, \quad n\geq1,\\
b_n=\frac{c_n}{c_{n-2}} b_{n-2} -g_n  \label{eqn:brzeqnforbnelim},\quad n\geq2, \\
\frac{c_n}{c_{n-2}}  =   \Bigl( \frac{c_{n-1}}{ c_{n-3}}\Bigr)^2 + 2  c_{n-1}g_{n-2}   -4c_{n-2} g_{n-1},\quad n\geq3. \label{eqn:brzeqnforc}
\end{gather}
\end{lemma}
\begin{proof}
The expression~\eqref{eqn:brzeqnfordnelim} is in the proof of Proposition~3.2 of~\cite{brzoska} for $n\geq 3$ but can be derived for $n\geq1$ by carefully applying Lemma~\ref{lem:DeterminantByRemovingVertices} to $B_n$ with $u$ a boundary vertex or can be checked from other results in~\cite{brzoska}: for $n=1$ use their equations~(3.4) and $d_1=1$ and for $n=2$ use their equations~(3.5) and $d_2=c_1=(\lambda-2)$.
The expression~\eqref{eqn:brzeqnforbnelim} is equation~(3.11) in the proof of Proposition~3.3 of~\cite{brzoska} where it is obtained by applying Lemma~\ref{lem:DeterminantByRemovingVertices} to a sequence of graphs not discussed in the present paper.  The expression~\eqref{eqn:brzeqnforc} is equation~(3.7) in~\cite{brzoska}.

We observe that the initial data and~\eqref{eqn:DefOfgn} allows us to compute all $c_n$. Then~\eqref{eqn:brzeqnforbnelim} yields all $b_n$ and we can use~\eqref{eqn:brzeqnfordnelim} to obtain all $d_n$.
\end{proof}

The purpose for which these and other identities were introduced in~\cite{brzoska} was to prove a recursion for the  polynomials $\gamma_n$ that could be used to study properties of the Dirichlet Laplacian spectra $\Dirspec(\Delta(G_n))$.  We note that it suffices to understand the roots of these polynomials, which are simple, because the multiplicities of the eigenvalues can be read from Theorem~\ref{thm:Brzfmlaforcn}.  The recursion from~\cite{brzoska} is stated in the following theorem both because it is used in later computations and in order that it may be compared to the recursion in Theorem~\ref{thm:mainrecursion}, which also includes the $\psi_n$ factors.

\begin{theorem}[{\protect \cite{brzoska}} Corollary~3.16]\label{thm:brzrecursion}
For $n\geq 4$ the polynomials $\gamma_n$ may be computed recursively as follows.  Define, for $n\geq1$,
\begin{equation}
	\zeta_n=\frac{\gamma_n}{\eta_n}, \text{ and }  \eta_n=\gamma_{n-1}\prod_{0\leq2j\leq{n-4}} \gamma_{n-2j-3}^{2^j}.
\end{equation}
Then $\zeta_n$ has the same roots as $\gamma_n$ and, for $n\geq4$, satisfies the recursion
\begin{equation}\label{eqn:brzrecursion}
\zeta_n-2 = \Bigl(1+\frac2{\zeta_{n-1}} \Bigr) (\zeta_{n-2}^2 -4)
\end{equation}
with initial data
\begin{equation*}
	\zeta_2=\frac{\lambda^2-6\lambda+4}{\lambda-2}, \qquad \zeta_3=\frac{\lambda^4-12\lambda^3+42\lambda^2-44\lambda+8}{\lambda^2-6\lambda+4}.
\end{equation*}
The remaining $\gamma_n$ is $\gamma_1=(\lambda-2)$.
\end{theorem}

\begin{figure}
\begin{centering}
\includegraphics{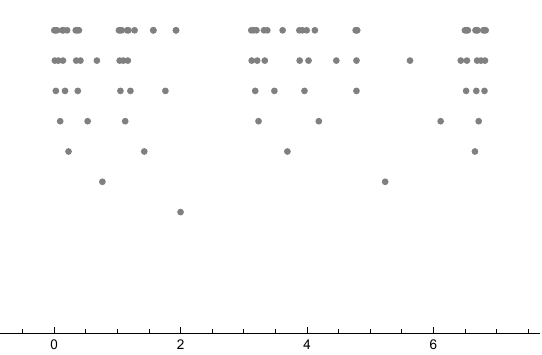}
\caption{The $2$-series eigenvalues of $\Gamma_{10}$, arranged so that the roots of $\gamma_{n-2}$ occur on the horizontal line at height $n$ for $n=0,\dotsc,10$ (the first three such lines being empty).}\label{fig:twoseries}
\end{centering}
\end{figure}

Using the recursion in Theorem~\ref{thm:brzrecursion} it is fairly quick to generate the eigenvalues in the $2$-series. Figure~\ref{fig:twoseries} shows all $2$-series eigenvalues of $\Gamma_{10}$ with each horizontal row corresponding to the level at which the eigenvalue was born. The first three rows are empty, as there are no $2$-series eigenvalues for $n=0,1,2$. The first occurrence of a $2$-series eigenvalue is at $n=3$, where we see the single eigenvalue $2$. Thereafter, the number of $2$-series eigenvalues at level $n$ is the degree of $\gamma_{n-2}$, which was determined in Proposition~3.17 of~\cite{brzoska}.

\subsection{Combined dynamics for $\gamma_n$ and $\psi_n$ factors}\label{subsec:Dynamics}

Having established the structure within which the dynamics for the factors $\gamma_n$ of the characteristic polynomials of $\Delta(G_n)$ can be understood and connected to the factors in the characteristic polynomial $P_n$ of $\Delta(\Gamma_n)$ we are in a position to calculate recursions that let us determine the roots of the $\psi_n$ factors in $P_n$.

\begin{proposition}\label{prop:CharPolyRecursion2}
The characteristic polynomials $P_n$ and $c_n$ may be obtained from the recursions
\begin{equation}\label{eq:P_nRecursionInC_n}
P_n=P_{n-1}\frac{c_n}{c_{n-2}}+2c_{n-1}g_{n-2}\biggl(\frac{c_n}{c_{n-2}}-2\frac{g_n}{g_{n-2}}\biggr), \quad n\geq 3
\end{equation}
and by using the initial data
\begin{equation*}
P_0 = \lambda,\quad
P_1 = \lambda^2 - 4\lambda,\quad
P_2 = \lambda^4-12\lambda^3+40\lambda^2-32\lambda.
\end{equation*}
\end{proposition}

\begin{proof}
If $n\geq2$ we can use~\eqref{eqn:brzeqnfordnelim}  to eliminate $d_n$ and $d_{n-1}$ from~\eqref{eq:CharPoly1} and write
\begin{equation*}
P_n = -3\lambda c_n c_{n-1} + 2c_{n-1}(b_n - g_n) + 2c_n(b_{n-1} - g_{n-1})
\end{equation*}
Since $n\geq3$ we can also apply this to $P_{n-1}$ and obtain
\begin{equation*}
P_{n-1}\frac{c_n}{c_{n-2}} = -3\lambda c_{n}c_{n-1} + 2c_n(b_{n-1}-g_{n-1}) + 2c_{n-1}\frac{c_n}{c_{n-2}}(b_{n-2}-g_{n-2}),
\end{equation*}
Their difference is then
\begin{equation*}
P_n-P_{n-1}\frac{c_n}{c_{n-2}} = 2c_{n-1}(b_n-g_n)-2c_{n-1}\frac{c_n}{c_{n-2}}(b_{n-2}-g_{n-2}).
\end{equation*}
and simplifying using~\eqref{eqn:brzeqnforbnelim} (applied to $b_n$, so $n\geq 3$ is sufficient here) recovers~\eqref{eq:P_nRecursionInC_n}.  The initial polynomials may be computed directly from the Laplacians of the initial graphs.
\end{proof}

Before we derive recursions for $\gamma_n$ and $\psi_n$ we pause to record a useful identity.

\begin{lemma}\label{lem:C_nG_nToZeta}
For $n \geq 2$, we have the identity
\begin{equation*}
\frac{c_n}{c_{n-2}} \cdot \frac{g_{n-2}}{g_n} = \zeta_n.
\end{equation*}
\end{lemma}
\begin{proof}
All terms are products of powers of $\gamma_k$.  From Theorem~\ref{thm:brzrecursion} we have $\frac{c_n}{c_{n-2}}=\gamma_n\prod_1^{n-3} \gamma_k^{\alpha_k}$ with $\alpha_k=S_{n-k}-S_{n-k-2}=\frac13(2^{n-k-2}-(-1)^{n-k})$. The computation for the $g_n$ factors is more involved, but in the proof of Proposition~3.14 of~\cite{brzoska} it was shown that $\frac{g_n}{g_{n-2}} =\gamma_{n-1}\prod_1^{n-3}\gamma_k^{\beta_k}$ with 
\begin{equation*}
	\beta_k = \begin{cases} \frac13\bigl(2^{(n-k-2)}+1\bigr)+2^{(n-k-3)/2} &\text{ if $n-k$ is odd,}\\
			\frac13\bigl( 2^{(n-k-2)} -1\bigr)  &\text{ if $n-k$ is even}.
			\end{cases}
	\end{equation*}
Checking that $\alpha_k-\beta_k = - 2^{(n-k-3)/2}$ when $n-k$ is odd and zero when $n-k$ is even and comparing this to the definition of $\zeta_n$ in Theorem~\ref{thm:brzrecursion} gives the desired result.
\end{proof}

Given that the recursion from~\cite{brzoska} in Theorem~\ref{thm:brzrecursion} provides the $\zeta_n$ polynomials for all $n$ we can compute the roots of the polynomials $\psi_n$ using the following result.
\begin{theorem}\label{thm:mainrecursion}
For $n\geq 5$ the rational function $\Zeta_n=\frac{\psi_n\gamma_{n-2}}{\gamma_n}$ satisfies the recursion
\begin{equation}
 \Zeta_n-1 = \frac1{\Zeta_{n-1}}\biggl(1-\frac1{\Zeta_{n-2}}\biggr)\biggl(1-\frac2{\zeta_n}\biggr)\frac1{\zeta_{n-2}-2}
\end{equation}
where the equality is valid at the poles in the usual sense of rational functions and the initial data is
\begin{align*}
\Zeta_3&=\frac{(\lambda^3-10\lambda^2+24\lambda-8)(\lambda-2)}{\lambda^4-12\lambda^3+42\lambda^2-44\lambda+8}\\
\Zeta_4&=\frac{(\lambda^5-16\lambda^4+88\lambda^3-192\lambda^2+136\lambda-16)(\lambda^2-6\lambda+4)}{\lambda^7-22\lambda^6+186\lambda^5-756\lambda^4+1508\lambda^3-1344\lambda^2+448\lambda-32}.
\end{align*}
The remaining $\psi_n$ are $\psi_1=\lambda(\lambda-4)$ and $\psi_2=\lambda^2-8\lambda+8$.
\end{theorem}
\begin{proof}
Moving both the $P_{n-1}$ terms of~\eqref{eq:P_nRecursionInC_n} to the left side of the expression and rewriting the last term on the right using Lemma~\ref{lem:C_nG_nToZeta} gives, for $n\geq3$, that
\begin{equation*}
P_n- P_{n-1}\frac{c_n}{c_{n-2}} = \frac{2c_nc_{n-1}g_{n-2}}{c_{n-2}} \Bigl(1-\frac2{\zeta_n}\Bigr).
\end{equation*}
Substituting for the $P_n$ and $P_{n-1}$ terms using the expression in Proposition~\ref{prop:Pnusingcn} gives
\begin{equation*}
	\biggl(\prod_{k=0}^n \psi_k\biggr) \frac{c_nc_{n-1}}{\gamma_n\gamma_{n-1}} -  \biggl(\prod_{k=0}^{n-1} \psi_k\biggr) \frac{c_nc_{n-1}}{\gamma_{n-1}\gamma_{n-2}} 
	= \frac{2c_nc_{n-1}g_{n-2}}{c_{n-2}}  \Bigl(1-\frac2{\zeta_n}\Bigr).
	\end{equation*}
which, after cancellation of $c_nc_{n-1}$ and multiplying by $\gamma_{n-1}\gamma_{n-2}$ is
\begin{equation*}
	\biggl(\prod_{k=0}^{n-1} \psi_k\biggr)(\Zeta_n-1)
	=\biggl(\prod_{k=0}^{n-1} \psi_k\biggr) \biggl(\frac{\psi_n\gamma_{n-2}}{\gamma_n} -1 \biggr)
	= \frac{2g_{n-2}\gamma_{n-1}\gamma_{n-2}}{c_{n-2}}  \Bigl(1-\frac2{\zeta_n}\Bigr). 
	\end{equation*}
Dividing this expression by the same expression with $n$ replaced by $n-2$, which is possible since we have assumed $n\geq 5$, we obtain
\begin{equation*}
	\psi_{n-1}\psi_{n-2} \biggl(\frac{\Zeta_n-1}{\Zeta_{n-2}-1} \biggr)
	= \frac{g_{n-2} c_{n-4}\gamma_{n-1}\gamma_{n-2} }{g_{n-4}c_{n-2}\gamma_{n-3}\gamma_{n-4}} \biggl(\frac{\zeta_n-2}{\zeta_{n-2}-2}\biggr)\frac{\zeta_{n-2}}{\zeta_n}
	\end{equation*}
whereupon moving the $\gamma$ factors to the left side and combining them with the $\psi$ factors yields the product $\Zeta_{n-1}\Zeta_{n-2}$, while Lemma~\ref{lem:C_nG_nToZeta} gives that the remaining term on the right is $\zeta_{n-2}^{-1}$ so we have
\begin{equation}\label{eqn:mainrecursion1}
	\Zeta_{n-1}\Zeta_{n-2} \biggl(\frac{\Zeta_n-1}{\Zeta_{n-2}-1} \biggr)
	=\frac{1}{\zeta_n}  \biggl(\frac{\zeta_n-2}{\zeta_{n-2}-2}\biggr).\qedhere
\end{equation}
\end{proof}

\begin{figure}
\begin{centering}
\includegraphics{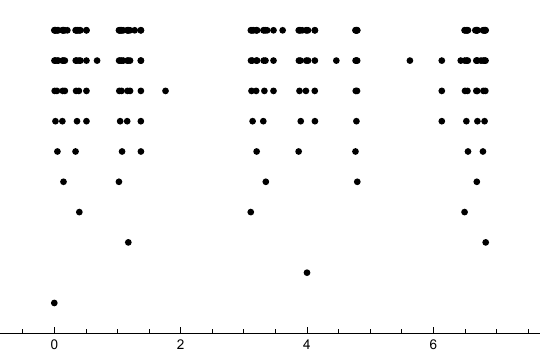}
\caption{The $0$-series eigenvalues of $\Gamma_{10}$, arranged so that the roots of $\psi_n$ occur on the horizontal line at height $n$ for $n=0,\dotsc,10$.}\label{fig:zeroseries}
\end{centering}
\end{figure}

\begin{corollary}\label{cor:0seriesfromzeta=2}
Let $\tilde\lambda$ be a $0$-series eigenvalue born at level $n$. Then there is $m$ so that $\zeta_n(\tilde\lambda)=2$ for all $k\geq m$.  If $n\geq 3$ then $m=n+2$. If $n=0$ or $n=1$ then $m=n+3$. If $n=2$ then $m=2$.  Moreover, in all cases except $n=2$ we also have $\zeta_{m-1}(\tilde\lambda)=-2$.  Hence the $0$-series eigenvalues can be obtained as backward orbits of $2$ under the dynamical system from Theorem~\ref{thm:brzrecursion}.
\end{corollary}
\begin{proof}
Suppose $\psi_n(\tilde\lambda)=0$. Observe that then $\psi_m(\tilde\lambda)\neq0$ for $m\neq n$ and, since $\tilde\lambda$ is a $0$-series eigenvalue and these are distinct from the $2$-series eigenvalues, we have $\gamma_m(\tilde\lambda)\neq0$. It follows from the definition of $\zeta_m$ that $\zeta_m(\tilde\lambda)\neq0$ and $\zeta_m(\tilde\lambda)\neq\infty$ for all $m$ and from the definition of $\Zeta_m$ that $\Zeta_m(\tilde\lambda)\neq0$ if $m\neq n$ while $\Zeta_n(\tilde\lambda)=0$.

Suppose $n\geq3$ so we can use~\eqref{eqn:mainrecursion1} for $n+2$ and the preceding to obtain
\begin{equation*}
	0=\Zeta_{n+1}\Zeta_{n} \biggl(\frac{\Zeta_{n+2}-1}{\Zeta_{n}-1} \biggr)
	=\frac{1}{\zeta_{n+2}}  \biggl(\frac{\zeta_{n+2}-2}{\zeta_{n}-2}\biggr).
	\end{equation*}
Since the $0$-series eigenvalues are simple by Lemma~\ref{lem:0seriessimple}, the zero in this expression is simple.  We deduce that $\zeta_{n+2}(\tilde\lambda)=2$ and, if $\zeta_n(\tilde\lambda)=2$ then the degree of the zero of $\zeta_{n+2}-2$ is one greater than the degree of the zero of $\zeta_n-2$ at $\tilde\lambda$.  Since we also have the recursion~\eqref{eqn:brzrecursion} we can write
\begin{equation*}
	\zeta_{n+2}-2 = \Bigl( 1+\frac2{\zeta_{n+1}}\Bigr) (\zeta_n^2-4)
	\end{equation*}
and deduce that $\zeta_{n+1}(\tilde\lambda)=-2$. This recursion then also gives $\zeta_{n+k}(\tilde\lambda)=2$ for all $k\geq3$.

If $m$ is the smallest integer so $\zeta_m=2$ then~\eqref{eqn:mainrecursion1} for $n=m$ says
\begin{equation*}
	\Zeta_{m-1}\Zeta_{m-2} \biggl(\frac{\Zeta_{m}-1}{\Zeta_{m-2}-1} \biggr)=0	
	\end{equation*}
There are three possibilities, the first being that  $\Zeta_{m-2}=0$, so $\psi_{m-2}(\tilde\lambda)=0$ and $m-2=n$, in which case the first time $\zeta_m(\tilde\lambda)=2$ is $m=n+2$ and since $\zeta_{n+1}(\tilde\lambda)=-2$ we have our desired conclusion. 
The second possibility is that $\Zeta_{m-1}=0$, but then reasoning as above we have $\zeta_{m+1}=2$ and $\zeta_{m}=-2$, in contradiction to our choice of $m$.  The final possibility is that $\Zeta_{m}=1$. In this case we use our choice of $m$ and the recursion to write
\begin{equation*}
	\Zeta_{m-k-1}\Zeta_{m-k-2} \biggl(\frac{\Zeta_{m-k}-1}{\Zeta_{m-k-2}-1} \biggr) \neq 0	
	\end{equation*}
for $k=2,4,\dotsc$.  Evidently this gives $\Zeta_{m-2j}=1$ for those $j$ for which the recursion is valid, which is until the first time $m-2j<5$. Accordingly either $\Zeta_3=1$ or $\Zeta_4=1$. Using the expressions for these in Theorem~\ref{thm:mainrecursion} and the initial data in Theorem~\ref{thm:brzrecursion} the former translates to $\gamma_2=0$, and the latter to $\gamma_1\gamma_3=0$, neither of which can be true at $\tilde\lambda$ since the roots of the latter are $2$-series eigenvalues.

Having established the result for $n\geq3$ we consider the cases $n=0,1,2$ and compute the values of $\zeta_m(\tilde\lambda)$ using  the formulas in Theorem~\ref{thm:brzrecursion} and $\zeta_1(\lambda)=\gamma_1(\lambda)=\lambda-2$.

When $\psi_0(\tilde\lambda)=0$ we have $\tilde\lambda=0$, $\zeta_1(\tilde\lambda)=-2$,  $\zeta_2(\tilde\lambda)=-2$, and $\zeta_3(\tilde\lambda)=2$, so applying the recursion gives $\zeta_m(\tilde\lambda)=2$ for all $m\geq 3$.

When  $\psi_1(\tilde\lambda)=0$ we have $\tilde\lambda=4$ so $\zeta_1(\tilde\lambda)=2$. We also compute that both $\zeta_2(\tilde\lambda)=-2$ and $\zeta_3(\tilde\lambda)=-2$, so $\zeta_m(\tilde\lambda)=2$ for all $m\geq 4$.

Finally, $\psi_2(\tilde\lambda)=\tilde\lambda^2-8\tilde\lambda+8=0$, so $\tilde\lambda=4\pm2\sqrt2$ and $\zeta_1(\tilde\lambda)$ is irrational. However, this may be rewritten as $\gamma_2(\tilde\lambda)=2\gamma_1(\tilde\lambda)$, or $\zeta_2(\tilde\lambda)=2$. We can also compute $\zeta_3(\tilde\lambda)=2$ (this is most easily done by reducing the numerator and denominator modulo $\tilde\lambda^2-8\tilde\lambda+8$).  It follows that $\zeta_m(\tilde\lambda)=2$ for all $m\geq2$.
\end{proof}

Figure~\ref{fig:zeroseries} shows all $0$-series eigenvalues of $\Gamma_{10}$ with each horizontal row corresponding to the level at which the eigenvalue was born, so that the $n^\text{th}$ row contains the roots of $\psi_n$, $n=0,1,\dotsc,10$.  The number of $0$-series eigenvalues was determined in Corollary~\ref{cor:deg(psi_n)}.  We compare the locations of the $0$-series and $2$-series eigenvalues in Figure~\ref{fig:bothseries}, which contains both series by level in the same manner.

\begin{figure}
\begin{centering}
\includegraphics{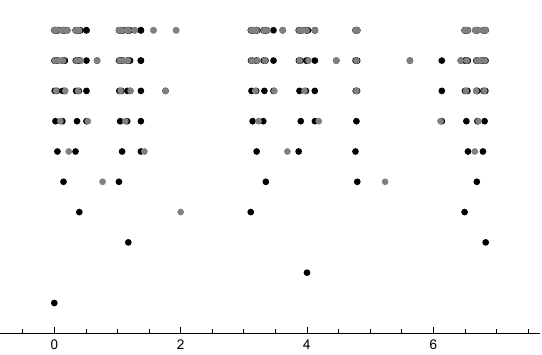}
\caption{The eigenvalues of $\Gamma_{10}$ with $2$-series in grey and $0$-series in black, arranged so that the horizontal line at height $n$ shows the eigenvalues born at level $n$ for $n=0,1,\dotsc,10$. There is considerable overlap of points on the lines for larger $n$.}\label{fig:bothseries}
\end{centering}
\end{figure}

\section{Spectral Measures}\label{sec:SpectralMeasure}
\begin{definition}\label{def:SpectralMeasure}
    The spectral counting measure of an operator $O$ is the sum of weighted point masses at each element of the spectrum of $O$:
 \begin{equation*}
 	\nu_n=\frac1{|\spec(O)|} \sum_{ \spec (O) } \delta_\lambda
 	\end{equation*}
\end{definition}

It was observed by Grigorchuk and \.Zuk (Theorem 4.2 and Corollary 4.3 in~\cite{GZ04}) that the spectral counting measures of the Markov operators corresponding to a convergent sequence of finite regular graphs, each of
which covers the next, converge weakly to a limit measure $\mu$ that they call the Kesten-von Neumann-Serre (KNS) measure.  See~\cite{GZ04} for a discussion of the connections between this measure and
the spectral measures studied by Kesten, the trace in the von Neumann
algebra generated by the left regular representation, and Serre's notion
of equidistribution with respect to $\mu$.  The $\Gamma_n$ are $4$-regular and $\Gamma_{n+1}$ covers $\Gamma_n$ for each $n$, and moreover these graphs converge in the sense of Section 3 in \cite{GZ04}, so the weak
limit of the spectral counting measures for the Laplacians $\Delta(\Gamma_n)$ is the KNS measure in this setting. 


\begin{definition}\label{def:KNSSpectralMeasure}
    The KNS spectral measure is the weak limit of the sequence of spectral counting measures.
    \begin{equation*}
    \mu=\underset{n\to\infty}{\textup{w-lim}}\, \nu_n.
    \end{equation*}
\end{definition}

\begin{remark}
It is a  standard result that the eigenvalues of the Laplacian on a $d$-regular graph with unit weight edges are in $[0,2d]$ (see~Section~1.3 of~\cite{ChungSpecGrphTh}), so the eigenvalues of $\Delta(\Gamma_n)$ lie in $[0,8]$ and, in particular, weak and vague convergence of probability measures coincide for our problem.  We use this fact without further comment.
\end{remark}

\begin{lemma}\label{lem:SpectralMeasureofGamma}
The spectral counting measure of $\Delta(\Gamma_n)$ is 
\begin{equation*}
	\nu_n
	=\frac1{2^n} \sum_{k=0}^n \biggl( \sum_{\{\lambda:\gamma_{k-2}(\lambda)=0\}} \sigma_{n-k} \delta_\lambda+\sum_{\{\lambda:\psi_k(\lambda)=0\}} \delta_\lambda \biggr)
	\end{equation*}
\end{lemma}
\begin{proof}
From  Lemma~\ref{lem:SizeofGamma} the number of vertices in $\Gamma_n$ is $2^n$ and thus $|\spec(\Delta(\Gamma_n))|=2^n$. Since $0$-series eigenvalues are simple (Lemma~\ref{lem:0seriessimple}) they occur with weight $2^{-n}$ in the sum from Definition~\ref{def:SpectralMeasure} for $\Delta(\Gamma_n)$. The multiplicity of each $2$-series eigenvalue born at level $k\leq n$ and occurring in $\Delta(\Gamma_n)$ is $\sigma_{n-k}$ (from Lemma~\ref{lem:2SeriesMultiplicity}) and was simple at level $k$ (by Corollary~\ref{lem:2SeriesSimpleBirth}) so occurs once as a root for $\gamma_{k-2}$. The formula follows.
\end{proof}

From the multiplicities and the degree of $\gamma_k$ we can estimate the mass in $\nu_n$ at the eigenvalues that occur as roots of $\gamma_k$. These roots are eigenvalues associated to the $2$-series,
\begin{lemma}\label{lem:2SeriesSpectralMeasureContribution}
\begin{equation*}
\Bigl| \sigma_{n-k}2^{-n}-\frac{2}{3} 2^{-k}\Bigr|\leq\frac23 2^{-n}.
\end{equation*}
\end{lemma}
\begin{proof}
Using Lemma~\ref{lem:2SeriesMultiplicity} we have
    \begin{align*}
        \left| \sigma_{n-k}2^{-n}-\frac{2}{3}2^{-k}\right|
        =&\left|\frac16 (2^{n-k+2}+3+(-1)^{n-k+1}) 2^{-n}-\frac{2}{3}2^{-k}\right|\\
        =&\frac16\left| 3+(-1)^{n-k+1}\right|2^{-n}
        \leq \frac{2}{3}2^{-n}
    \end{align*}
\end{proof}

This tells us that for fixed $k$ and large $n\gg k$, the spectral counting measure $\nu_n$ has atoms of approximately weight $\frac{2}{3}\cdot 2^{-k}$ at each eigenvalue in the $2$-series of  $\Delta(\Gamma_k)$.  We can also use the degree of $\psi_k$ to estimate the total weight of $0$-series eigenvalues.  Recall that the latter was computed in terms of the roots $\rho_j$ of a certain polynomial in Corollary~\ref{cor:deg(psi_n)}.
\begin{lemma}\label{lem:0SeriesSpectralMeasureContribution}
The spectral mass from the $0$-series eigenvalues in $\Delta(\Gamma_n)$ is 
\begin{equation*}
	2^{-n}  \bigl(\deg(\gamma_n)+\deg(\gamma_{n-1})-1\bigr)
	=  \frac{2^{1-n}}{\sqrt{7}} \sum_{j=1,2,3} (\rho_j^n+\rho_j^{n-1})\cos\bigl(\phi+\frac{2\pi j}3\bigr)
	\leq 20 (0.95)^n
	\end{equation*}
so the spectral mass of the $0$-series eigenvalues in $\Delta(\Gamma_n)$ converges exponentially to zero as $n\to\infty$.
\end{lemma}
\begin{proof}
The $0$-series eigenvalues are simple (Lemma~\ref{lem:0seriessimple}) and are the roots of the $\psi_k$, from which the spectral mass of these is simply $2^{-n}\sum_0^n\deg(\psi_k)$. In Corollary~\ref{cor:deg(psi_n)} we saw $\deg(\psi_n)=\deg(\gamma_n)-\deg(\gamma_{n-2})$ for $n\geq2$, so the sum is telescoping and the spectral mass is
\begin{equation}\label{eqn:0SeriesSpectralMeasureContribution1}
	2^{-n} \bigl(\deg(\gamma_n)+\deg(\gamma_{n-1})-\deg(\gamma_1)-\deg(\gamma_0)\bigr)
	=2^{-n}  \bigl(\deg(\gamma_n)+\deg(\gamma_{n-1})-1\bigr)
	\end{equation}
where we have used $\gamma_0=1$ and $\gamma_1=c_1=(\lambda-2)$.  Applying the same substitution for the values of $\deg(\gamma_n)$ as in Corollary~\ref{cor:deg(psi_n)} we find this to be
\begin{equation*}
	 \frac{2^{1-n}}{\sqrt{7}} \sum_{j=1,2,3} (\rho_j^n+\rho_j^{n-1})\cos\bigl(\phi+\frac{2\pi j}3\bigr)
	\end{equation*}
From Proposition~3.17 in~\cite{brzoska} we have a crude estimate
\begin{equation}\label{eqn:0SeriesSpectralMeasureContribution2}
	\deg (\gamma_n )
	\leq\frac{6}{\sqrt{7}} \bigl(\max\{|\rho_1|,|\rho_2|,|\rho_3|\}\bigr)^n
	\leq\frac{6}{\sqrt{7}} (1.9)^n
	\end{equation}
and substitution into~\eqref{eqn:0SeriesSpectralMeasureContribution1} provides a bound of $18\bigl(\frac{1.9}{2}\bigr)^{n-1}$ which is smaller than the given bound.
\end{proof}

\begin{theorem}\label{thm:KNSSpectralMeasureofGamma}
The KNS spectral measure for the sequence $\Delta(\Gamma_n)$ is 
\begin{equation*}
	\mu= \frac23 \sum_{k=3}^\infty \ 2^{-k}\hspace{-1em} \sum_{\{\lambda:\gamma_{k-2}(\lambda)=0\}} \delta_\lambda.
	\end{equation*}
\end{theorem}
\begin{proof}
Suppose $f$ is continuous and bounded by $M$ on the closure of $\cup_n(\Delta(\Gamma_n))$.  Showing $\int f d\nu_n\to\int f d\mu$ is the same as showing that
\begin{equation*}
	\sum_{k=3}^n \sum_{\{\lambda:\gamma_{k-2}(\lambda)=0\}} 2^{-n}\sigma_{n-k}f(\lambda) 
	+\sum_{k=0}^n \sum_{\{\lambda:\psi_k(\lambda)=0\}} 2^{-n} f(\lambda)
	-  \frac23 \sum_{k=3}^\infty \ 2^{-k}\hspace{-1em} \sum_{\{\lambda:\gamma_{k-2}(\lambda)=0\}} f(\lambda)
	\to0
	\end{equation*}
as $n\to\infty$. Note that we removed the terms for $k=1,2$ of the sum from $\nu_n$ because $\{\lambda:\gamma_{k-2}(\lambda)=0\}$ is empty for these $k$.

The middle term in the above expression corresponds to the $0$-series eigenvalues, so it is apparent from $|f(\lambda)|\leq M$ and Lemma~\ref{lem:0SeriesSpectralMeasureContribution} that this term converges to zero as $n\to\infty$. 

Combining the difference of the other two terms, using $|f(\lambda)|\leq M$ and Lemma~\ref{lem:2SeriesSpectralMeasureContribution}  and then estimating the size of $\{\lambda:\gamma_{k-2}(\lambda)=0\}=\deg(\gamma_{k-2})$ as in~\eqref{eqn:0SeriesSpectralMeasureContribution2} gives
\begin{align*}
	\lefteqn{ \biggl| \sum_{k=0}^n \sum_{\{\lambda:\gamma_{k-2}(\lambda)=0\}} (2^{-n}\sigma_{n-k} - \frac23 2^{-k}\bigr)f(\lambda)  - \sum_{n+1}^\infty \frac23 2^{-k}\bigr)f(\lambda) \biggr| } \quad&\\
	&\leq \frac{2M}3 \sum_{k=0}^n 2^{-n} \deg(\gamma_{k-2}) + \frac{2M}3\sum_{k=n+1}^\infty  2^{-k} \deg(\gamma_{k-2}) \\
	&\leq \frac{4M}{\sqrt{7}} \biggl( 2^{-n} \Bigl(\frac{(1.9)^{n+1}-1}{.9}\Bigr) + \sum_{k=n+1}^\infty \Bigl(\frac{1.9}{2}\Bigr)^k\biggr)\\
	&\leq C_1 \sum_n^\infty (0.95)^n = C_2 (0.95)^n 
	\end{align*}
for some constants $C_1$, $C_2$. This converges to zero as $n\to\infty$.
\end{proof}

\begin{theorem}\label{thm:SpectralMeasureIsAllSupportedOnHighLevel}
There is a constant $C$ so for any $\epsilon>0$, if $m\geq C(1+|\log \epsilon|)$ then all but $\epsilon$ of the spectral mass of any $\Gamma_n$, $n\geq m$, is supported on eigenvalues of the Laplacian of $\Gamma_m$. In particular this is true for the KNS measure.
\end{theorem}
\begin{proof}
The spectral mass of $\Delta(\Gamma_n)$ not supported on $\spec(\Delta(\Gamma_m))$ is 
\begin{align*}
\sum_{k=m+1}^n \bigl( 2^{-n}\sigma_{n-k}\deg(\gamma_{k-2}) + 2^{-n}\deg(\psi_k) \bigr)
	&\leq \sum_{k=m+1}^n \Bigl( \frac23 (2^{-k}+ 2^{-n}) \deg(\gamma_{k-2}) + 2^{-n} \deg(\gamma_{k}) \Bigr)\\
	&\leq \frac{12}{3\sqrt{7}} \sum_{k=m+1}^n\bigl( (2^{-k} + 2^{-n}) (1.9)^{k-2} + 2^{-n}  (1.9)^k\bigr)\\
	&\leq \frac{12}{\sqrt{7}} \sum_{k=m+1}^n  (0.95)^k  \leq C_1 (0.95)^m
	\end{align*}
where we used Lemma~\ref{lem:2SeriesSpectralMeasureContribution} to estimate $\sigma_{n-k}$ and  Corollary~\ref{cor:deg(psi_n)} and~\eqref{eqn:0SeriesSpectralMeasureContribution2} to estimate $\deg(\psi_k)$. 
\end{proof}

\begin{remark}\label{rem:dynamicaldegreecontrolsconvergence}
It should be noted that precise estimates on the rate of convergence of the spectral counting measures to the KNS measure may be obtained in terms of the roots $\rho_j$ from Corollary~\ref{cor:deg(psi_n)}, and could be estimated using the largest, $\rho_3$. The former is done explicitly for the $0$-series spectral mass in Lemma~\ref{lem:0SeriesSpectralMeasureContribution} and could be done in the proof of Theorem~\ref{thm:KNSSpectralMeasureofGamma} by inserting the formula for the degree of the $\gamma_{k-2}$ from Corollary~\ref{cor:deg(psi_n)} rather than estimating by~\eqref{eqn:0SeriesSpectralMeasureContribution2}.  Similarly, the value of the constant $C$ in Theorem~\ref{thm:SpectralMeasureIsAllSupportedOnHighLevel} could be estimated in terms of $\rho_3$.  Since $\rho_3$ is the dynamical degree (see Remark~\ref{rem:dynamicaldegree}) this says that the rate of convergence of the spectral counting measures to the KNS measure is controlled by the dynamical degree.
\end{remark}

In~\cite{brzoska} it is noted that the sequence of graphs $G_n$ does not exactly fit within the class in~\cite{GZ04} for which Grigorchuk and \.Zuk defined the KNS measure. It was argued there that since the weak limit of the spectral counting measures of the Dirichlet Laplacians on the $G_n$ had many of the same structural features as considered in~\cite{GZ04} this limit should still be considered a KNS measure, and that this KNS measure is given by (\cite{brzoska} Corollary~4.3)
\begin{equation}\label{eqn:brzKNSmsr}
	\chi=\sum_{k=1}^\infty\sum_{\{\lambda:\gamma_k\left(\lambda\right)=0\}}\frac{1}{6}\cdot2^{-k}\cdot\delta_\lambda
	\end{equation}

\begin{corollary}[Equivalence of KNS Spectral Measures]\label{cor:chi=Chi}
    The KNS measures for the sequences $\Delta(\Gamma_n)$ and the Dirichlet Laplacian $\Delta(G_n)$ are the same.
\end{corollary}
\begin{proof}
This is simply a matter of changing variables in the summation~\ref{eqn:brzKNSmsr} and using that $\{\lambda:\gamma_k\left(\lambda\right)=0\}$ is empty when $k=0,-1,-2$ to see the result is the same as the formula in Theorem~\ref{thm:KNSSpectralMeasureofGamma}.
\end{proof}

\begin{figure}
\begin{centering}
\includegraphics[width=7cm]{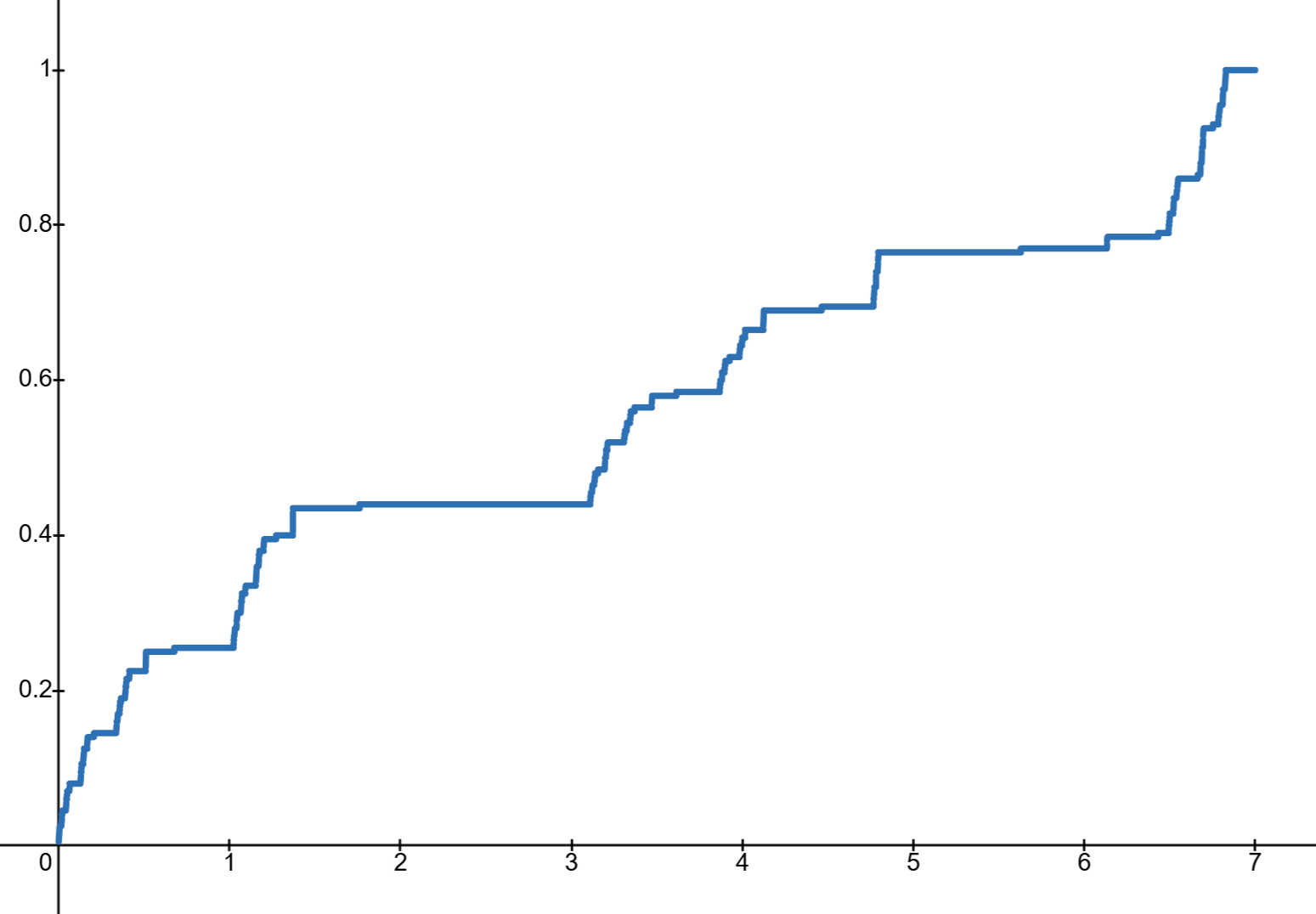}
\hspace{1.5cm}
\includegraphics[width=7cm]{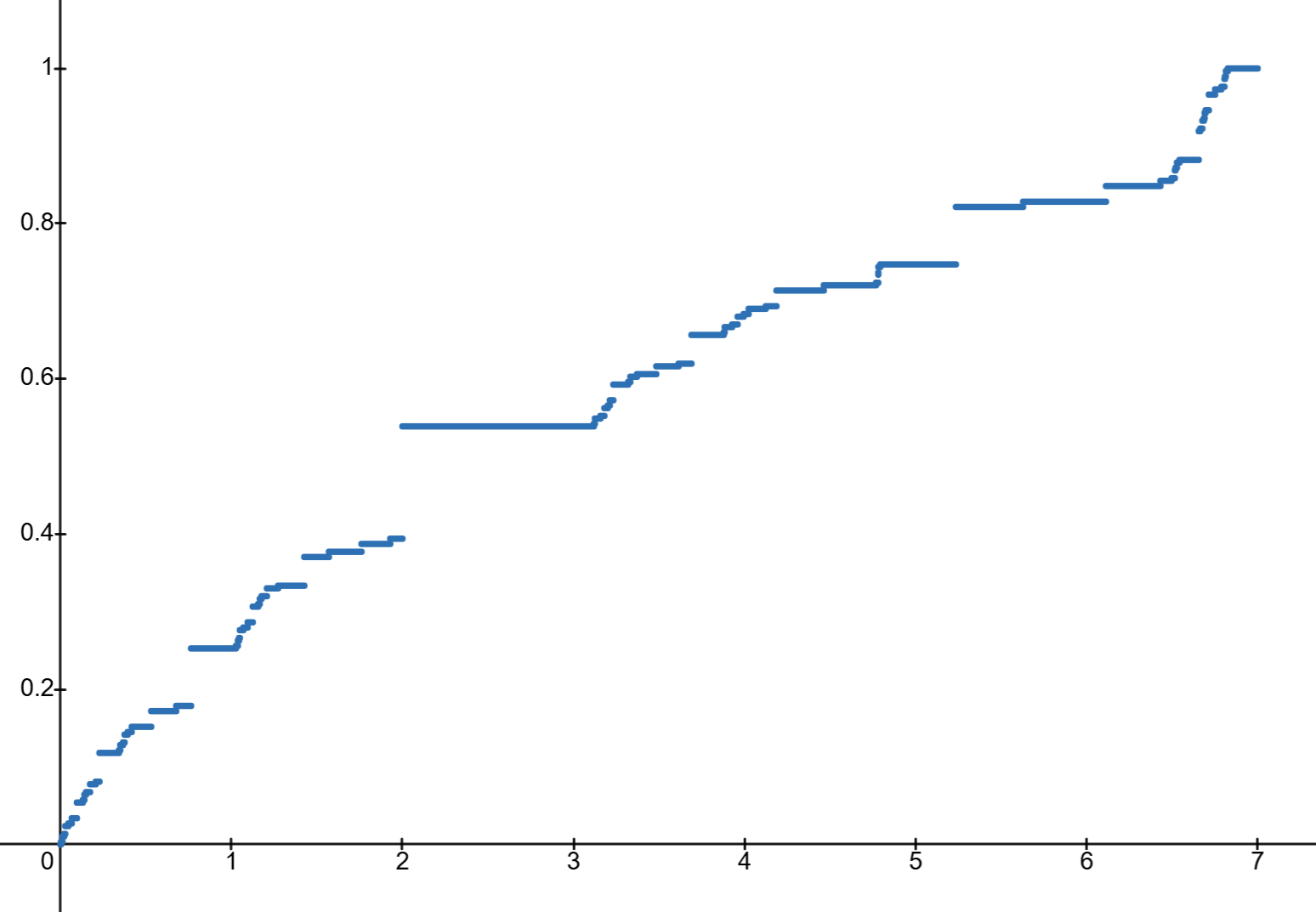}
\caption{Spectral counting functions for the $0$-series eigenvalues (left) and $2$-series eigenvalues (right) of  $\Delta(\Gamma_9)$.}
\label{fig:countingfunctions}
\end{centering}
\end{figure}

The KNS measure is purely atomic -- it is a countable union of point mass measures.  However, there are a number of standard ways to further elucidate its structure.  One comes from examining the cumulative distribution functions for the measures $\nu_n$. These are usually called spectral counting functions, as the value at $x$ counts the number of eigenvalues less than or equal to $x$.  In Figure~\ref{fig:countingfunctions} we show graphs of the spectral counting functions for both the $0$-series and $2$-series eigenvalues of $\Delta(\Gamma_n)$ with $n=9$.  We have obtained these graphs for other values of $n$, but by $n=9$ the changes with increasing $n$ are no longer visible.

A feature that is easily seen in these graphs is the existence of spectral gaps, which are simply intervals that contain no eigenvalues. They show as horizontal parts of the graph where the spectral counting function is constant. It is trivial that such exist for finite $n$, but it appears that they persist as $n\to\infty$.  This commonly occurs when the limiting spectral measure is invariant under a dynamical system, as it is usual that then the eigenvalues accumulate to the Julia set and the gaps are the Fatou components.

Another visible feature is the jumps at $2$-series eigenvalues. This is a consequence of  Lemma~\ref{lem:SpectralMeasureofGamma}, which establishes that a positive proportion of the spectral mass occurs at each $2$-series eigenvalue, and in fact that in the limit $n\to\infty$ this proportion depends only on the level of birth.

Yet a further feature of the graphs is that at each $2$-series eigenvalue there appears to be a spectral gap on one side of the eigenvalue and a sequence of further eigenvalues that converge on the other side of the eigenvalue; perhaps the spectral counting function is even continuous on one side. We also notice that the spectral counting function for the $0$-series eigenvalues may converge to a continuous function as $n\to\infty$. This is consistent with the fact that the $0$-series eigenvalues are simple (Lemma~\ref{lem:0seriessimple}) but the number of $0$-series eigenvalues has power-law growth, as established in Corollary~\ref{cor:deg(psi_n)}, but a proof would require information about the distribution of these eigenvalues, so as to ensure there is no point where a positive proportion of their spectral counting measure accumulates.

\begin{figure}
\begin{centering}
\includegraphics[width=7cm]{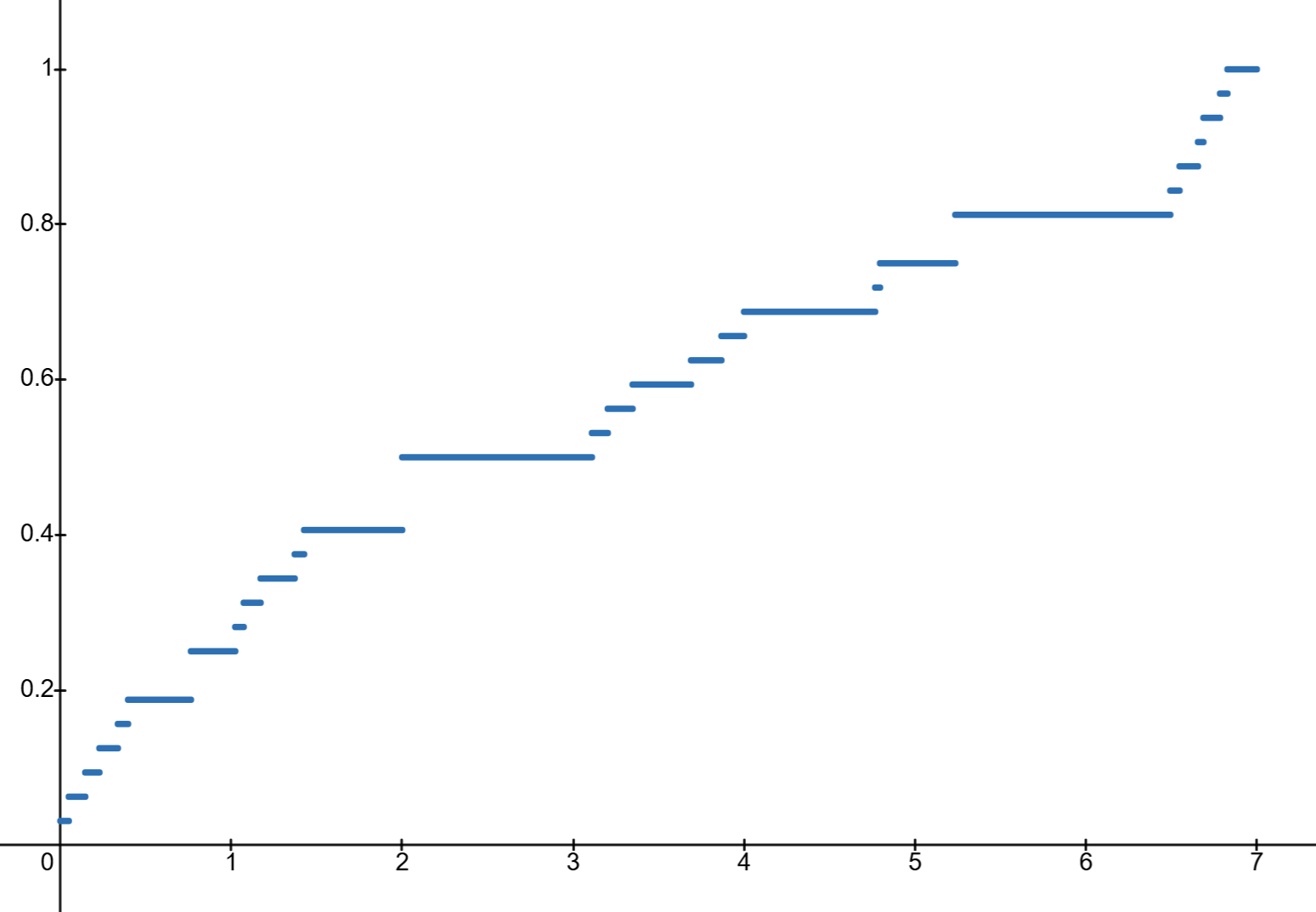}
\hspace{1.5cm}
\includegraphics[width=7cm]{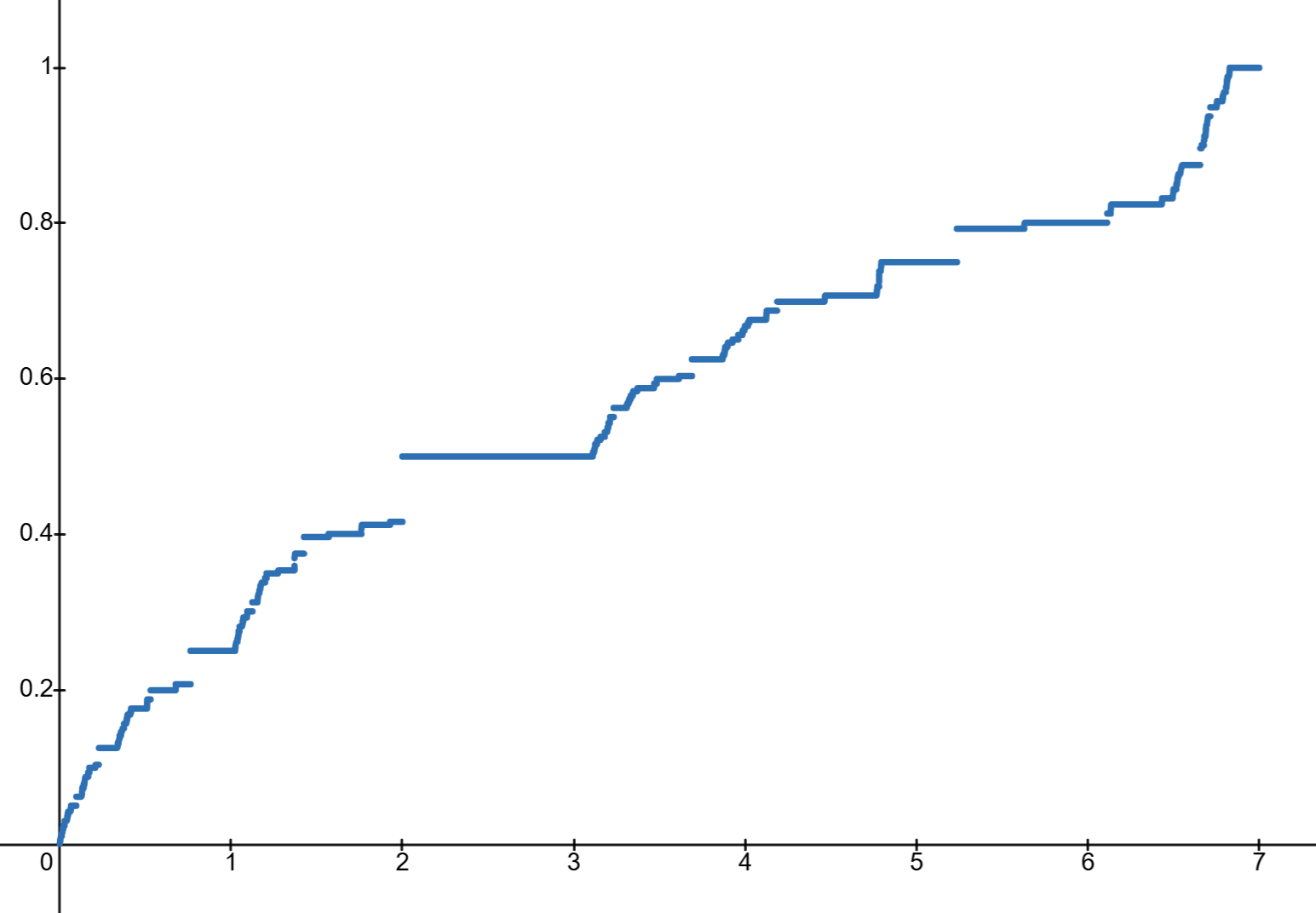}
\caption{Spectral counting function for $\Delta(\Gamma_n)$  with $n=5$ (left) and $n=9$ (right).}
\label{fig:bothseriescounting}
\end{centering}
\end{figure}

These three features are also visible on the graphs of the spectral counting function for all eigenvalues of $\Delta(\Gamma_n)$, which is shown for $n=5$ and $n=9$ in Figure~\ref{fig:bothseriescounting}. At $n=5$ the graphs are still visibly changing with increasing $n$, while from $n=9$ onward there is little visible change.  The fact that there are jumps at $2$-series eigenvalues is an immediate consequence of the formula in Theorem~\ref{thm:KNSSpectralMeasureofGamma}.  The existence of gaps in the $2$-series spectrum and that every $2$-series eigenvalue is a one-sided limit point of the set of all $2$-series eigenvalues was proved in~\cite{brzoska}. As a consequence of this and Theorem~\ref{thm:SpectralMeasureIsAllSupportedOnHighLevel} we have that the support of the KNS measure, which is the closure of the set $\cup_n \Dirspec(\Delta(G_n))$, is a Cantor set. We understand that it follows from the work of Dang, Grigorchuk and Lyubich that this Cantor set has positive Lebesgue measure, though the proof has not yet appeared.

\begin{corollary}\label{cor:CantorArgument}
The support of the KNS spectrum of \(\Delta\left(\Gamma_n\right)\) is a Cantor set. In particular it is uncountable and has countably many gaps.
\end{corollary}

However, there is another spectral set that naturally arises in the study of $\Gamma_n$, namely the closure of $\cup_n \spec(\Delta(\Gamma_n))$. A priori this could be a larger set than the support of the KNS spectrum, as the former is the closure of all eigenvalues and the latter the closure of just the $2$-series eigenvalues. However, the dynamical description obtained in this paper allows us to prove that these sets are the same.

\begin{theorem}\label{thm:accumsetofevals}
The closure of the set of eigenvalues $\cup_n \spec(\Delta(\Gamma_n))$ is the support of the KNS measure, so is a Cantor set.
\end{theorem}
\begin{proof}
Evidently it suffices to prove that every $0$-series eigenvalue is a limit of $2$-series eigenvalues. Fix some $0$-series eigenvalue $\tilde\lambda$.  From Corollary~\ref{cor:0seriesfromzeta=2} we know there is $m$ so that $\zeta_n(\tilde\lambda)=2$ for all $n\geq m$. Taking $m$ to be as small as possible we also know $m\geq 2$ and that $\zeta_{m-1}(\tilde\lambda)=-2$ unless $m=2$.  Also recall that the roots and poles of the functions $\zeta_n$ occur exactly at $2$-series eigenvalues.  We proceed in the manner of the proof of Lemma~5.2 of~\cite{brzoska}, showing by contradiction that any sufficiently small interval around $\tilde\lambda$ must contain a root or pole of some $\zeta_n$.

Suppose there is an interval $I\ni\tilde\lambda$ containing no roots or poles of any $\zeta_n$, $n\geq m$. Then, since each $\zeta_n(\tilde\lambda)=2$, we have that each $\zeta_n$ is continuous and positive on $I$. The recursion from Theorem~\ref{thm:brzrecursion},
\begin{equation*}
	\zeta_{n+2}-2 = \biggl(1+\frac2{\zeta_{n+1}}\biggr)(\zeta_n^2-4),
	\end{equation*}
is valid for $n\geq m$ because $m\geq2$, and our preceding reasoning says the first factor on the right is greater than $1$.  We check that if we can find $n\geq m$ and $\lambda$ such that $\zeta_n(\lambda)=2-\delta_n<2$ then we arrive at a contradiction.  The reason is that from $\zeta_n>0$ we have $\delta_n<2$ and thus $\zeta_n^2-4<-2\delta_n$, so that $\zeta_{n+2}<2-2\delta_n$.  Iteration then produces some $j$ so $\delta_{n+2j}(\lambda)<0$, in contradiction to all such being positive on $I$.

It remains to verify that such a $\lambda$ exists, which we do for $n=m$. In essence this is because the $0$-series eigenvalues are simple. In the proof of  Corollary~\ref{cor:0seriesfromzeta=2} this fact was used to see that $\zeta_{n+2}-2$ has a simple zero at $\tilde\lambda$ when the level of birth  $n$ is at least $3$, and in this case $m=n+2$.  For $0$-series eigenvalues born at levels $0,1,2$ we just check directly that the corresponding functions $\zeta_m-2$, which are $\zeta_3-2$, $\zeta_4-2$ and $\zeta_2-2$ respectively, also have a simple zero at the eigenvalue $\tilde\lambda$.  Evidently, the presence of a simple zero at $\tilde\lambda\in I$ when $I$ is open implies there is $\lambda\in I$ with $\zeta_m(\lambda)<2$, at which point the above reasoning guarantees that $I$ also contains a zero or pole of $\zeta_{m+2j}$ for some sufficiently large $j$.	
\end{proof}

A further observation from the graphs of the spectral counting functions has to do with ``gap labels'' for the gaps in the spectrum. The gap label for a particular gap interval is simply the value of the counting function on the interval, which is just the mass of the spectral counting measure that lies to the left of the gap. In the spectral counting graphs shown in Figure~\ref{fig:bothseriescounting} it appears that the gap labels are dyadic: they have the form $j2^{-k}$.  It is often possible to identify the gap labels of spectra invariant under a dynamical system, but here the fact that our dynamics is for the points where $\zeta_n=0$ and $\zeta_n=2$ makes this approach difficult.

\begin{conjecture}
The gap labels for the KNS measure are dyadic rationals.
\end{conjecture}

\providecommand{\bysame}{\leavevmode\hbox to3em{\hrulefill}\thinspace}
\providecommand{\MR}{\relax\ifhmode\unskip\space\fi MR }
\providecommand{\MRhref}[2]{%
  \href{http://www.ams.org/mathscinet-getitem?mr=#1}{#2}
}
\providecommand{\href}[2]{#2}

\end{document}